\documentclass[authoryear,preprint,11pt,3p]{elsarticle} 
\journal{European Journal of Operations Research}

\makeatletter
\def\els@aparagraph[#1]#2{\elsparagraph[#1]{#2\@addpunct{.}}}
\def\els@bparagraph#1{\elsparagraph*{#1\@addpunct{.}}}
\makeatother

\usepackage{amsmath,amsfonts,amssymb}
\usepackage{mathtools}
\usepackage{bm}
\usepackage{breqn}
\usepackage{hyperref}
\usepackage{siunitx}

\usepackage[dvipsnames]{xcolor}
\usepackage{cleveref}

\hypersetup{
    colorlinks=true,
    linkcolor=blue,
    citecolor=ForestGreen,
    urlcolor=blue
}

\crefformat{algorithm}{#2{\color{BrickRed}#1}#3}
\Crefformat{algorithm}{#2{\color{BrickRed}#1}#3}
\usepackage{mathrsfs}

\usepackage{amsthm}
\theoremstyle{plain}%
\newtheorem{theorem}{Theorem}

\DeclareMathOperator*{\argmin}{arg\,min}
\DeclareMathOperator\supp{supp}

\usepackage{rotating}
\usepackage{longtable}
\usepackage{booktabs}
\usepackage{threeparttable}

\usepackage{algpseudocode}
\usepackage{algorithm}
\algrenewcommand{\algorithmicindent}{0.7em}  

\DeclareMathOperator{\RU}{RU}
\DeclareMathOperator{\RD}{RD}
\DeclareMathOperator{\UT}{UT}
\DeclareMathOperator{\DT}{DT}
\DeclareMathOperator{\RAE}{RAE}

\usepackage{pgfplots}
\pgfplotsset{compat=newest}
\usepgfplotslibrary{statistics}
\usepackage{xcolor}
\usepackage[export]{adjustbox} 
\usepackage{pgfplotstable} 
\usetikzlibrary{calc}

\usepackage{caption} 
\makeatletter
\let\@msm@th@eqref\eqref
\renewcommand{\eqref}[1]{%
  \begingroup
  \leavevmode
  \color{blue}%
  \hypersetup{linkbordercolor=[named]{blue}}%
  \@msm@th@eqref{#1}%
  \endgroup
}
\makeatother

\usepackage{setspace}

\begin{document}

\begin{frontmatter}
\title{Scenario Reduction for Two-Stage Stochastic Mixed-Integer Programs}
\author[inst1]{Yannick Werner\corref{cor1}}\ead{yannick.werner@tugraz.at}
\author[inst2]{Juan Miguel Morales}\ead{juan.morales@uma.es}
\author[inst2]{Salvador Pineda}\ead{spineda@uma.es}
\author[inst3]{Line Roald}\ead{roald@wisc.edu}
\author[inst1]{Sonja Wogrin}\ead{wogrin@tugraz.at}

\affiliation[inst1]{
    organization={Graz University of Technology},
    city={Graz},
    country={Austria}
}
\affiliation[inst2]{
    organization={University of Málaga},
    city={Málaga},
    country={Spain}
}
\affiliation[inst3]{
    organization={University of Wisconsin-Madison},
    city={Madison},
    country={USA}
}
\cortext[cor1]{Corresponding author.}

\begin{abstract}
Two-stage stochastic mixed-integer programs are important tools for decision-making under uncertainty. Representing the uncertainty with many scenarios, however, can make them challenging to solve. Scenario reduction addresses this by finding a distribution supported on fewer scenarios that still yields similar optimal first-stage decisions. In this paper, we revisit the classical scenario reduction theory based on distances between probability distributions and the optimal mass transportation problem. The transportation problem's cost function captures scenario similarity and is central to the effectiveness of scenario reduction. We then review and compare various transportation cost functions from the literature and propose a new one. Using the Forward Selection Algorithm, we prove that our proposed cost function selects the best possible scenario from a given sample on the first draw with respect to the relative approximation error. To reduce the computational cost of evaluating this cost function, we further propose a hybrid algorithm with a scenario pre-selection phase. We assess solution quality and computational complexity on the two-stage stochastic unit commitment problem for small 24-bus and large 300-bus case studies. With only around five scenarios, the proposed cost function approximates the full-distribution optimum to within roughly \qty{2.1}{\percent} and \qty{0.4}{\percent} error for the small and large cases, respectively. In contrast, prevalent cost functions often need 25 scenarios or more to achieve that solution quality. The hybrid algorithm achieves similar solution quality while reducing wall-clock time by a factor of 18 and work (per Gurobi solver) by a factor of 66 on the large case study.
\end{abstract}


\begin{highlights}
\item We propose a new cost function for the optimal mass transportation problem.
\item Forward Selection is proven to select the optimal scenario first with that function.
\item We compare methods on stochastic unit commitment systems with up to 300 buses.
\item Our cost function yields \qty{2.4}{\percent} and \qty{0.4}{\percent} error with 1 and 5 scenarios on 300 buses.
\item  A hybrid algorithm cuts solver work 66x on 300 buses at indistinguishable accuracy.
\end{highlights}


\end{frontmatter}
\section{Introduction}

Two-stage stochastic mixed-integer programs (MIPs) have become a popular tool for supporting decision-making under uncertainty across a wide range of applications, including supply chain network design~\citep{Santoso2005}, integrated staffing and scheduling~\citep{Kim2015}, and the traveling salesperson problem~\citep{Cavaliere2026}.
Frequently, the uncertainty is approximated using a sufficiently large set of scenarios, motivated by theoretical advances in the sample average approximation \citep{Kleywegt2002} that provide mathematical guarantees on the asymptotic behavior of scenario-based approximations of continuous distributions \citep{Shapiro2021}. At the same time, using a large set of scenarios drastically increases the complexity of the stochastic MIPs~\citep{Klein2021}, potentially making them intractable or very challenging to solve, necessitating modern solution methods such as decomposition~\citep{Romeijnders2026}.
 
Another way to overcome this problem is to apply scenario reduction techniques, which aim to reduce the number of scenarios needed to accurately represent uncertainty in the stochastic program without compromising solution accuracy. 
Finding accurate reduced sets of scenarios supported on only a few atoms is particularly crucial for two-stage stochastic MIPs, where solution time increases exponentially with the number of scenarios.
Pioneering work by~\citet{Dupacova2003} relates the efficiency of scenario reduction to the distance between the probability distributions of the underlying random variables, as measured by the optimal mass transportation problem. They derive mathematical stability results and provide algorithms for scenario selection, most notably the Forward Selection Algorithm, which we recap later. This work was extended by~\citet{Heitsch2003} to improve the computational efficiency of the proposed algorithms and by~\citet{Heitsch2007} to consider different probability metrics. While these earlier works largely addressed underlying stability and convergence properties (see also \citet{Pflug2001}, \citet{Roemisch2003}, and \citet{Roemisch2007}), the applied distance measures only assess the similarity of input data and do not take the structure of the optimization problem into account, leading to slow convergence rates \citep{Henrion2018}. We refer to those cost functions as \textit{input-data-driven}.

More recent literature has focused on incorporating information from the stochastic optimization problem into the scenario reduction process. We refer to those as \textit{optimization-problem-driven}. In their seminal work, \citet{Morales2009} propose a scenario distance metric based on the optimal objective values of the second-stage recourse problems obtained by fixing the first-stage decision variables to their optimal values in the expected value problem.
\citet{Bruninx2016}, however, argue that this may lead to risk-averse decision-making and propose using the difference in the objective function value of a single scenario in deterministic two-stage problems instead. More recently, \cite{Bertsimas2023} propose a cost function that takes the loss of decision quality into account when making first-stage decisions based on a specific scenario, while another one is eventually realized.

Based on those works, we identify several gaps in the literature that challenge a rigorous comparison of different cost functions for the mass transportation problem. First, some papers proposing optimization-problem-driven cost functions or scenario reduction methods, such as \citet{Bertsimas2023} and \citet{Hewitt2021}, respectively, only compare their approach to input-data-driven methods. Second, papers that provide more in-depth comparisons, such as \citet{Dvorkin2014} and \citet{Zhuang2025}, often vary both distance measures and selection procedures simultaneously. Finally, some papers proposing optimization-problem-driven distances that require solving many small deterministic problems only consider convex \citep{Bertsimas2023} or small mixed-integer \citep{Zhuang2025} programs, thereby limiting evaluation of scalability and potentially making those methods impractical for large-scale problems. 

Connecting to the classical literature on scenario reduction based on the distance of probability distributions and the optimal mass transportation problem, we overcome the aforementioned gaps by the following contributions:
\begin{itemize}
    \item We compare several optimization-problem-driven cost functions for the optimal mass transportation problem and benchmark them against the most prominent input-data-driven one, using the Forward Selection Algorithm~\citep{Dupacova2003}, facilitating a \emph{ceteris paribus} comparison.
    \item We propose a new cost function for the mass transportation problem and prove a theorem stating that the first scenario drawn from a discrete distribution using this cost function and the Forward Selection Algorithm is optimal in terms of the relative approximation error.
    \item Using computationally challenging case studies, we compare all cost functions by evaluating the quality of the reduced scenario sets and the combined computational effort required to obtain them and to solve the corresponding reduced two-stage stochastic MIP.
    \item We present a novel yet simple hybrid algorithm that combines the cost function by~\citet{Morales2009} with the proposed one, achieving exceptional approximation quality while remaining computationally efficient, even for large-scale two-stage stochastic MIPs.
\end{itemize}

The methods in this paper are developed for and applicable to general scenario-based two-stage stochastic MIPs with uncertainty in the right-hand sides of the second-stage constraints. To demonstrate their performance, we consider the exemplary two-stage stochastic unit commitment (SUC) problem, as it has been well studied in the literature \citep{Takriti1996, Zheng2015,Haaberg2019}, and several of the proposed scenario reduction techniques \citep{Papavasiliou2013,Feng2014,Dvorkin2014,Bruninx2016,Zhuang2025} have been applied to it. The full mathematical formulation of the two-stage SUC problem is provided in~\ref{sec:app_SUC}, while we focus on generic formulations in the main body of the paper.

The remainder of this document is organized as follows. Section~\ref{sec:preliminaries} introduces a generic scenario-based two-stage stochastic MIP. Afterward, Section~\ref{sec:ScenarioReduction} recaps the classical theory on scenario reduction and the Forward Selection Algorithm proposed by~\citet{Dupacova2003}. Section~\ref{sec:CostFunctionMTP} introduces the different cost functions and derives the optimality theorem for the proposed one.
Performance results for the various cost functions applied to the two-stage stochastic unit commitment problem are presented and discussed in Section~\ref{sec:results}. Finally, Section~\ref{sec:conclusions} concludes.
\section{Preliminaries}
\label{sec:preliminaries}

In this paper, we consider two-stage stochastic MIPs with uncertainty only in the right-hand sides of the second-stage constraints, represented by the random variable $\xi$. We further assume that $\xi$ can be adequately represented by a finite set of scenarios following the discrete distribution $\mathbb{P} = \sum_{i \in \mathcal{I}} p_i \delta_{\xi_i}$, where $p_i$ and $\xi_i \in \mathbb{R}^d$ are the probability and location of scenario $i$ for $i \in \mathcal{I} = \{1,\dots,n\}$ under distribution $\mathbb{P}$ and $\delta_{\xi_i}$ is the Dirac distribution allocating unit mass at $\xi_i$. Under these assumptions, we formulate the following generic two-stage stochastic MIP:
\begin{equation}\label{eq:GeneralTwoStageProblem}
    \min_{x \in \mathcal{X}} z(x, \xi) = f(x) + \sum_{i \in \mathcal{I}} p_i G(x, \xi_i),
\end{equation}
where $G(x, \xi_i)$ is given by:
\begin{subequations}\label{eq:GeneralSecondStage}
\begin{align}
    \min_{v} \quad & a^\top v \\
    \mathrm{s.t.} \quad & W v \geq h^{\xi_i} - T x, \label{eq:GeneralSecondStageConstraints}\\
    & v \geq 0.
\end{align}
\end{subequations}
The goal of Problem~\eqref {eq:GeneralTwoStageProblem} is choosing decision variables $x \in \mathcal{X}$, where $\mathcal{X}$ is defined by a set of mixed-integer linear constraints, to minimize first-stage cost $f(x)$ plus the expected second-stage cost taken with respect to the distribution of the random variable $\xi$, i.e., $\mathbb{E}_{\xi \sim \mathbb{P}}[G(x, \xi)] =  \sum_{i \in \mathcal{I}} p_i G(x, \xi_i)$, while the second stage problem~\eqref{eq:GeneralSecondStage} chooses continuous actions $v$ to minimize the linear operational cost $a^\top v$ subject to a set of linear constraints~\eqref{eq:GeneralSecondStageConstraints} that depend on the first-stage decisions $x$ and the realization $\xi_i$ of the random variable $\xi$. Notably, we take the following two common assumptions: (1) The cost function coefficients $a$, the recourse matrix $W$, and the technology matrix $T$ are independent of the uncertainty (note that the latter is generally not needed), and (2) the second-stage problem~\eqref{eq:GeneralSecondStage} is feasible for any realization $x \in \mathcal{X}$ (relatively complete recourse). For a discussion of those assumptions and the two-stage stochastic program in general, we refer the reader to~\cite{Birge2011} and \cite{Shapiro2021}.

We use $x^*(\mathbb{P})$ and $z^*(\mathbb{P})$ to denote the optimal first-stage decisions and objective function value, respectively, of Problem~\eqref{eq:GeneralTwoStageProblem} with respect to distribution $\mathbb{P}$:
\begin{subequations}\label{eq:TwoStageOpt}
\begin{align}
    x^*(\mathbb{P}) = \arg&\min_{x \in \mathcal{X}} f(x) + \sum_{i \in \mathcal{I}} p_i G(x, \xi_i) \label{eq:TwoStageOptSol}\\
    z^*(\mathbb{P}) = &\min_{x \in \mathcal{X}} f(x) + \sum_{i \in \mathcal{I}} p_i G(x, \xi_i). \label{eq:TwoStageOptObj}
\end{align}
\end{subequations}

Additionally, we define for any given $\hat{x} \in \mathcal{X}$, $z^*(\hat{x},\mathbb{P})$, i.e., the optimal value function of Problem~\eqref{eq:GeneralTwoStageProblem} with first-stage decisions fixed at $\hat{x}$, that is:
\begin{equation}
    z^*(\hat{x},\mathbb{P}) = f(\hat{x}) + \sum_{i \in \mathcal{I}} p_i G(\hat{x}, \xi_i).
\end{equation}

For the scenario reduction later on, we further define a deterministic version of Problem~\eqref{eq:GeneralTwoStageProblem} considering a distribution supported on a single scenario $\xi_i$ only:
\begin{equation}\label{eq:GeneralSingleScenarioDP}
    \min_{x \in \mathcal{X}} z(x,\xi_i) = f(x) + G(x, \xi_i),
\end{equation}
where we denote the corresponding optimal first-stage decision as:
\begin{equation}\label{eq:SingleScenarioOptSol}
    x^*(\xi_i) = \arg\min_{x \in \mathcal{X}} z(x,\xi_i).
\end{equation}

\section{A recap on scenario reduction}
\label{sec:ScenarioReduction}
The first thing we want to point out is that the real difficulty in stochastic programming in general is finding the optimal first-stage decisions $x^*(\mathbb{P})$. Once those are found, $z^*(\mathbb{P})$ can be easily determined by evaluating the individual, independent second-stage problems~\eqref{eq:GeneralSecondStage}, as, for any given $x^*$ (in fact, for any feasible $\hat{x} \in \mathcal{X}$), Equation~\eqref{eq:TwoStageOptObj} in combination with \eqref{eq:GeneralSingleScenarioDP} simplifies to:
\begin{equation}\label{eq:SecondStageEvaluationDistribution}
    z^*(x^*,\mathbb{P}) = \sum_{i \in \mathcal{I}} p_i (f(x^*) + G(x^*, \xi_i)) = \sum_{i \in \mathcal{I}} p_i z(x^*, \xi_i).
\end{equation}
Now, suppose there is another distribution $\mathbb{Q} = \sum_{j \in \mathcal{J}} q_j \delta_{\zeta_j}$, where $q_j$ and $\zeta_j \in \mathbb{R}^d$ are the probability and location of scenario $j$ for $j \in \mathcal{J} = \{1,\dots,m\}$ under distribution $\mathbb{Q}$ with $m < n$, and ideally $m \ll n$, as well as $x^*(\mathbb{Q}) \approx x^*(\mathbb{P})$.
Then we would not need to solve the optimization problem~\eqref{eq:GeneralTwoStageProblem} under the distribution $\mathbb{P}$, which can be intractable or computationally challenging, but instead under $\mathbb{Q}$, which is much easier. The goal of scenario reduction, which we discuss next, is to find such a distribution $\mathbb{Q}$.

\subsection{Distance of probability distributions}
\label{sec:ProbabilityDistributionDistance}
One common way in the mathematical programming literature to assess the similarity of probability distributions $\mathbb{P}$ and $\mathbb{Q}$ is to measure their distance $\mathfrak{D}$ using the Monge-Kantorovich mass transportation problem (see e.g. \citet{Heitsch2003}):
\begin{equation}\label{eq:mass_transport}
\begin{aligned}
\mathfrak{D}(\mathbb{P}, \mathbb{Q}) = \min_{\pi \in \mathbb{R}_{+}^{n \times m}} \quad & \sum_{i=1}^{n} \sum_{j=1}^{m} \pi_{ij} c(\xi_i, \zeta_j) \\
\text{s.\,t.} \quad & \sum_{j=1}^{m} \pi_{ij} = p_i, \quad \forall i \in \mathcal{I}, \\
& \sum_{i=1}^{n} \pi_{ij} = q_j, \quad \forall j \in \mathcal{J},
\end{aligned}
\end{equation}
where $c(\xi_i, \zeta_j)$ is some function that measures the cost of moving the probability mass from scenario $\xi_i$ under distribution~$\mathbb{P}$ to $\zeta_j$ under distribution~$\mathbb{Q}$.
Let the support of distribution~$\mathbb{P}$ be defined as $\supp(\mathbb{P}) = \{\xi_i : i \in \mathcal{I} \}$ and by $\mathcal{P}(\Xi,m)$, for any set $\Xi \in \mathbb{R}^d$, the family of discrete distributions supported on $m$ points (scenarios) in $\Xi$. Then the goal of the \textit{discrete scenario reduction problem} is to find a distribution $\mathbb{Q} \in \mathcal{P}(\supp(\mathbb{P}),m)$, with $m < n$, and ideally $m \ll n$, such that both distributions $\mathbb{P}$ and $\mathbb{Q}$ are close with respect to some distance, e.g., as defined here in Equation~\eqref{eq:mass_transport} (cf. \citet{Rujeerapaiboon2018}). From now on, we restrict ourselves to the discrete scenario reduction problem in this paper (i.e., the atoms in the reduced distribution $\mathbb{Q}$ must be drawn from the support of $\mathbb{P}$), and therefore use $\xi_j$ instead of $\zeta_j$ in the following to enhance notational clarity.

\subsection{Forward Selection Algorithm}
\label{sec:ForwardSelectionAlgorithm}
In the seminal work on scenario reduction for stochastic programming using probability metrics, \citet{Dupacova2003} prove that in the specific case where the reduced distribution $\mathbb{Q}$ is only supported on atoms of $\mathbb{P}$, such that $\mathcal{J} \subset \mathcal{I}$, that for any $\mathcal{J}$, there is an analytical solution to Problem~\eqref{eq:mass_transport}, which only depends on $p_i$ and $c(\xi_i,\xi_j)$ [\cite{Dupacova2003}, Theorem 3.1]:
\begin{equation}\label{eq:ProbDist_set}
    \mathfrak{D}(\mathbb{P}, \mathbb{Q}) = \mathfrak{D}(\mathbb{P}, \mathcal{J}) \coloneq \sum_{i \in \mathcal{I} \setminus \mathcal{J}} p_i \min_{j \in \mathcal{J}} c(\xi_i, \xi_j),
\end{equation}
with
\begin{subequations}
\begin{align}
    q_j &= p_j + \sum_{i \in \mathcal{I}_j} p_i, \mathrm{where} \label{eq:ProbabiltyRedistribution}\\
    \mathcal{I}_j &\coloneq \{ i \in \mathcal{I}\,\setminus\,\mathcal{J} : j = \arg\min_{j' \in \mathcal{J}} ~ c(\xi_i, \zeta_{j'})\}.
\end{align}
\end{subequations}
Expression~\eqref{eq:ProbabiltyRedistribution} is referred to as the optimal probability redistribution rule. Intuitively, the redistribution of probabilities is such that the probability mass $p_i$ of all scenarios in $\mathcal{I}$ under distribution $\mathbb{P}$, which are closest to scenario $j$ with respect to $c(\xi_i, \xi_j)$, is aggregated. The theorem implies that the goal of scenario reduction is actually finding the index set $J$, since the corresponding optimal probability distribution~$\mathbb{Q}$ follows immediately.

Based on that theorem, \citet{Dupacova2003} propose the Forward Selection Algorithm for which a pseudo-code is provided in Algorithm~\ref{alg:ForwardSelection}. We briefly recap the algorithm here for completeness, as it will be relevant to the derivation of our theoretical result later. To ease notation, we introduce an auxiliary set~$\mathcal{R}$ that captures all unselected scenarios from the original sample. Intuitively, the algorithm iteratively populates the set of scenarios $\mathcal{J}$ by a single scenario $j$ based on how the distance $\mathfrak{D}(\mathbb{P}, \mathcal{J})$, which is dependent on the choice of cost function $c(\xi_i,\xi_j)$, would change if scenario $j$ were added to the reduced set of scenarios, until $m$ scenarios are selected in total. For clarity, we will use $\mathbb{P}_n$ and $\mathbb{Q}_m$ from here on to highlight the number of atoms the distributions are supported on.

While the Forward Selection Algorithm constitutes a greedy strategy, it is proven in~\cite{Dupacova2003} that it provides a globally optimal solution (selection) for $m=1$. We acknowledge that computational improvements to the Forward Selection Algorithm have been suggested in the literature, e.g., in~\cite{Heitsch2003}, but resort to the simpler initial version here for clarity.
%

\begin{algorithm}
\setstretch{1.15}
\caption{Forward Selection \citep{Dupacova2003}}
\begin{algorithmic}[1]
\Statex \textbf{inputs} $\mathbb{P}_n$, $m$, $c(\xi_i,\xi_j)$
\State set $\mathcal{J} \gets \{\}$
\While{$|\mathcal{J}| < m$}
    \State set $\mathcal{R} \gets \mathcal{I} \setminus \mathcal{J}$
    \State $\displaystyle j = \operatorname{arg}\min_{j' \in \mathcal{R}} \sum\nolimits_{i \in \mathcal{R} \setminus \{ j' \}} p_i \min_{j'' \in \mathcal{J} \cup \{j'\} } c(\xi_i,\xi_{j''}) $
    \State $\mathcal{J} \gets \mathcal{J} \cup \{j\}$
\EndWhile
\For{$j \in \mathcal{J}$}
    \State $ \mathcal{I}_j \coloneq \{ i \in \mathcal{I}\,\setminus\,\mathcal{J} : j = \argmin_{j' \in \mathcal{J}} ~ c(\xi_i, \xi_{j'}) \}$ 
    \State $ \text{update } q_j = p_j + \sum\nolimits_{i \in \mathcal{I}_j} p_i$
\EndFor
\Statex \textbf{return} $\mathbb{Q}_m$
\end{algorithmic}
\label{alg:ForwardSelection}
\end{algorithm}

\subsection{Evaluation of scenario reduction methods}
While the Forward Selection Algorithm~\cref{alg:ForwardSelection} provides an efficient greedy heuristic to find a distribution~$\mathbb{Q}_m$, supported on $m$ scenarios, with a small distance~\eqref{eq:mass_transport}, its ability to accurately approximate the optimal decisions $x^*(\mathbb{P}_n)$ of Problem~\eqref{eq:GeneralSecondStage} hinges on the choice of the cost function $c(\xi_i,\xi_j)$. Before discussing different formulations of the cost function in Section~\ref{sec:CostFunctionMTP}, we briefly discuss how to evaluate the approximation quality of $\mathbb{Q}_m$ related to Problem~\eqref{eq:GeneralTwoStageProblem}.

For clarity, we define the optimal solution of the two-stage stochastic problem~\eqref{eq:GeneralTwoStageProblem} with respect to distribution $\mathbb{Q}_m$ in line with Equations~\eqref{eq:TwoStageOpt} as:
\begin{subequations}
\begin{align}
    x^*(\mathbb{Q}_m) &= \arg\min_{x \in \mathcal{X}} \sum_{j \in \mathcal{J}} q_j \cdot z(x,\xi_j), \label{eq:TwoStageOptSolQ}  \\
    z^*(\mathbb{Q}_m) &= \min_{x \in \mathcal{X}} \sum_{j \in \mathcal{J}} q_j \cdot z(x,\xi_j). \label{eq:TwoStageOptObjQ}
\end{align}
\end{subequations}

Building upon the definition in~\cite{Pflug2001} and \cite{kaut2003evaluation}, we then define the \textit{relative approximation error} $\RAE$ to quantify how well $\mathbb{Q}_m$ approximates $\mathbb{P}_n$ with respect to the two-stage stochastic problem~\eqref{eq:GeneralTwoStageProblem}:
\begin{align}\label{eq:RAE}
    & \RAE(\mathbb{P}_n,\mathbb{Q}_m) = \frac{z^*(x^*(\mathbb{Q}_m),\mathbb{P}_n) - z^*(\mathbb{P}_n)}{z^*(\mathbb{P}_n)}, \\
    & \mathrm{where}~z^*(x^*(\mathbb{Q}_m),\mathbb{P}_n) = \sum_{i \in \mathcal{I}} p_i \cdot z(x^*(\mathbb{Q}_m),\xi_i). \nonumber
\end{align}
The $\RAE$ can be interpreted as the \textit{regret} of making decisions based on distribution $\mathbb{Q}_m$ while the true distribution is $\mathbb{P}_n$. 
We point out that this metric is only well-defined when $z^*(\mathbb{P}_n) \neq 0$ and that it is always nonnegative for any $x \in \mathcal{X}$ by definition:
\begin{equation}
    z^*(\mathbb{P}_n) =  z^*(x^*(\mathbb{P}_n)), \mathbb{P}_n) \leq z^*(x, \mathbb{P}_n).
\end{equation}
This implies that there will always be a loss in decision quality when approximating distribution $\mathbb{P}_n$ with any distribution $\mathbb{Q}_m$ with respect to Problem~\eqref{eq:GeneralTwoStageProblem}, unless $x^*(\mathbb{Q}_m) = x^*(\mathbb{P}_n)$. Trivially, the $\RAE$ is zero when distributions $\mathbb{P}_n$ and $\mathbb{Q}_m$ are equal.

We acknowledge that evaluating metric~\eqref{eq:RAE} is generally impractical, as the underlying assumption in scenario reduction is that it is not possible to solve the two-stage stochastic problem~\eqref{eq:GeneralTwoStageProblem} for distribution $\mathbb{P}_n$ in the first place. This fact has given rise to research on stability theorems \citep{Pflug2001, Dupacova2003, Roemisch2003}, focusing on establishing upper bounds for this error based on the distance between the probability distributions $\mathbb{P}_n$ and $\mathbb{Q}_n$. As our main focus here is to evaluate the impact of the choice of cost function $c(\xi_i,\xi_j)$, we assume that we can actually solve Problem~\eqref{eq:GeneralTwoStageProblem} for comparison purposes.

\section{Cost functions for the transport problem}
\label{sec:CostFunctionMTP}
After recapping the scenario reduction process, we now want to take a look at various formulations for the cost function $c(\xi_i, \xi_j)$ of the mass transportation problem~\eqref{eq:mass_transport} suggested in the literature. Note that in the context of stochastic programming, ideally we would like to define the cost function $c$ and distance $\mathfrak{D}$ directly with respect to the optimal first-stage decisions $x^*(\mathbb{P})$ and $x^*(\mathbb{Q})$ of Problem~\eqref{eq:GeneralTwoStageProblem}. However, for real-world applications, this is impractical, as we cannot solve those problems and therefore do not know the decisions $x^*(\mathbb{P})$ (or even $x^*(\mathbb{Q})$ for every possible distribution $\mathbb{Q}$). Therefore, we need to find a measure $c$ that provides a good way to approximate the loss in decision quality when moving from $x^*(\mathbb{P})$ to $x^*(\mathbb{Q})$.

We can generally separate the formulations proposed in the literature into two categories. First, those measures that are solely based on the realizations (scenarios) of the random variables $\xi_i$, i.e., the values of $h^{\xi_i}$ in Problem~\eqref{eq:GeneralSecondStage} discussed in this paper, which we refer to as \textit{input-data-driven}. And, secondly, those functions that take some information from the stochastic optimization problem~\eqref{eq:GeneralTwoStageProblem} or some approximation of it into account. We refer to those as \textit{optimization-problem-driven}. In the following, we introduce a selection of the most commonly discussed cost functions $c(\xi_i,\xi_j)$.

\subsection{Input-data-driven}
In the context of input-data-driven measures, it has been proposed many times, e.g., in~\cite{Rujeerapaiboon2018}, to use the squared Euclidean norm of the scenario realizations:
\begin{equation}\label{eq:c_ID}
    c^{\mathrm{ID}}(\xi_i, \xi_j) = \lVert h^{\xi_i} - h^{\xi_j} \rVert_2^2.
\end{equation}
As shown, e.g., in~\cite{Pflug2011}, the resulting distance $\mathfrak{D}$ equals the squared type-2 Wasserstein distance. We refer the interested reader to~\citet{Rujeerapaiboon2018}, who provide an in-depth description and analysis of the metric and its application to scenario reduction. Here, we restrict ourselves to this distance within the class of input-data-driven measures, as it has also been used for comparison in other studies investigating optimization-problem-driven distances, such as \citet{Bertsimas2023}, and has been successfully applied to the two-stage SUC problem, e.g., in \citet{Dvorkin2014} and \citet{Zhuang2025}.

\subsection{Optimization-problem-driven}
In the following, we introduce several optimization-problem-driven transportation cost functions that, in contrast to the input-data-driven one in Equation~\eqref{eq:c_ID}, measure the similarity of scenarios by incorporating information related to the two-stage stochastic problem~\eqref{eq:GeneralTwoStageProblem}. In particular, they all rely on solving various versions and combinations of the single-scenario deterministic problem~\eqref{eq:GeneralSingleScenarioDP} and single-scenario second-stage problems~\eqref{eq:GeneralSecondStage} and therefore depend on the specific problem~\eqref{eq:GeneralTwoStageProblem} desired to solve.

\subsubsection{\citet{Morales2009}}
\citet{Morales2009} suggest measuring the similarity of scenarios based on the difference in second-stage costs when fixing the first-stage variables to those of the optimal solution of the expected value problem (EVP) \citep{Birge2011}:
\begin{equation}\label{eq:c_Mo}
    c^{\mathrm{Mo}}(\xi_i, \xi_j) = | z(x^*(\bar{\xi}), \xi_i) - z(x^*(\bar{\xi}), \xi_j) |,
\end{equation}
where $\bar{\xi} = \mathbb{E}_{\mathbb{P}_n}[\xi]$ is the expected realization of the random variable $\xi$ under distribution $\mathbb{P}_n$ and $x^*(\bar{\xi})$ is the optimal first-stage decision of the EVP.
When applied to two-stage stochastic MIPs, evaluating this measure requires solving a single MIP~\eqref{eq:GeneralSingleScenarioDP}, i.e., the EVP, and $n = |\mathcal{I}|$ second-stage LPs~\eqref{eq:GeneralSecondStage}.

\subsubsection{\citet{Bruninx2016}}
\citet{Bruninx2016} argue that the measure $c^{\mathrm{Mo}}$ may lead to an overly conservative, risk-averse selection of scenarios, as the first-stage solution to the EVP provides very little flexibility in extreme cases and therefore leads to high recourse costs in those cases. To overcome this problem, they suggest using the difference in optimal objective value of the single-scenario problems~\eqref{eq:GeneralSingleScenarioDP}: 
\begin{equation}\label{eq:c_Br}
    c^{\mathrm{Br}}(\xi_i, \xi_j) = | z(x^*(\xi_i), \xi_i) - z(x^*(\xi_j), \xi_j) |.
\end{equation}
This requires solving $n$ MIPs~\eqref{eq:GeneralSingleScenarioDP}.

\subsubsection{\citet{Bertsimas2023}}
Aligned with the stability theory presented in the classical literature on scenario reduction for convex programs using input-data-driven transportation cost functions \citep{Dupacova2003,Heitsch2003}, \citet{Bertsimas2023} propose the distance:
\begin{align}\label{eq:c_Be}
    2 \cdot c^{\mathrm{Be}}(\xi_i, \xi_j) = z&(x^*(\xi_j), \xi_i) - z(x^*(\xi_i),\xi_i) + \nonumber\\
    z&(x^*(\xi_i), \xi_j) - z(x^*(\xi_j),\xi_j),
\end{align}
which accounts for the symmetric loss in decision quality when making first-stage decisions based on scenario $\xi_i$ and then observing $\xi_j$, and vice versa. In order to evaluate this cost function, one needs to solve $n$ MIPs~\eqref{eq:GeneralSingleScenarioDP} and $n^2 - n$ second-stage LPs~\eqref{eq:GeneralSecondStage}.

\subsubsection{Our proposal}

In this paper, we propose the following novel cost function for the optimal mass transportation problem~\eqref{eq:mass_transport}:
\begin{equation}\label{eq:c_Pr}
    c^{\mathrm{Pr}}(\xi_i, \xi_j) = z(x^*(\xi_j), \xi_i) - z(x^*(\xi_i),\xi_i),
\end{equation}
which measures the \emph{asymmetric} regret incurred when the optimal first-stage decisions are determined based on scenario $\xi_j$, while the actual realization is $\xi_i$. It therefore implicitly captures the impact of the one-directional loss of information when approximating the original sample distribution $\mathbb{P}_n$ with the reduced distribution $\mathbb{Q}_m$.
Similar to $c^{\mathrm{Be}}$, evaluating $c^{\mathrm{Pr}}$ requires solving $n$ MIPs~\eqref{eq:GeneralSingleScenarioDP} and $n^2-n$ second-stage LPs~\eqref{eq:GeneralSecondStage}. The motivation for introducing this cost function is not merely heuristic. Rather, it is specifically designed to endow the classical scenario reduction framework with a theoretical property that, to the best of our knowledge, has not previously been established. In particular, combining cost function $c^{\mathrm{Pr}}$ with the optimal mass transportation formulation~\eqref{eq:mass_transport} and the Forward Selection Algorithm~\cref{alg:ForwardSelection} yields the following theorem, which constitutes the main theoretical contribution of this paper.

\begin{theorem}\label{th:Theorem1}
    Let $\mathcal{Q}_1 = \{ \mathbb{Q}_1^j : j \in \mathcal{I} \}$ denote the family of probability distributions obtained by drawing a single scenario from distribution $\mathbb{P}_n$, and $j^*$ be the first scenario drawn using the Forward Selection Algorithm~\cref{alg:ForwardSelection} and cost function $c^{\mathrm{Pr}}$. Then the distribution $\mathbb{Q}_1^{j^*}$ minimizes the $\RAE$ for $m = 1$, such that
    \[
    \mathbb{Q}_1^{j^*} = \arg\min_{\mathbb{Q}_1 \in \mathcal{Q}_1} \mathrm{RAE}(\mathbb{P}_n,\mathbb{Q}_1).
    \]
\end{theorem}
\textit{Proof:} See~\ref{sec:app_proof}.
\vspace{1mm}

\noindent Theorem~\eqref{th:Theorem1} provides the theoretical rationale behind cost function~\eqref{eq:c_Pr}. When $m=1$, minimizing the transportation cost induced by $c^{\mathrm{Pr}}$ is equivalent to minimizing the resulting $\RAE$. Consequently, the first scenario selected by the Forward Selection Algorithm~\cref{alg:ForwardSelection} is exactly the one that yields the smallest approximation error among all single-scenario distributions. Although this optimality guarantee is restricted to the first selected scenario, our computational experiments demonstrate that this first choice can already provide an accurate approximation to the optimal first-stage decisions for the full distribution and serves as an excellent starting point for constructing high-quality reduced distributions with $m>1$.

\subsection{Summary}
We summarize the required optimization problems that must be solved to evaluate each cost function $\mathcal{C} = \{ \mathrm{ID},\mathrm{Mo},\mathrm{Br}, \mathrm{Be}, \mathrm{Pr} \}$ in Table~\ref{tab:OptProbRequirments}.
\begin{table}
\centering
\caption{Comparison of optimization problems that must be solved to evaluate the cost function $c$.}
\label{tab:OptProbRequirments}
\begin{tabular}{lccccc}
\toprule
 & ID & Mo & Br & Be & Pr \\
\midrule
MIPs & $0$ & $1$ & $n$ & $n$ & $n$ \\
LPs & $0$ & $n$ & $0$ & $n^2 - n$ & $n^2 - n$ \\
\bottomrule
\end{tabular}
\end{table}
In Section~\ref{sec:large_case_comp_exp}, we present numerical experiments and a detailed analysis of the computational expenses of solving these problems for large-scale stochastic MIPs.

Algorithm~\cref{alg:ScenarioReduction} provides a pseudo-code for the whole scenario reduction process, including solving the reduced two-stage stochastic MIPs and evaluating the solution for each cost function $c \in \mathcal{C}$. Again, as a reminder, the goal of scenario reduction is to find a distribution $\mathbb{Q}_m$ based on an original distribution $\mathbb{P}_n$, with $m < n$ and ideally $m \ll n$, that provides a good approximation $x^*(\mathbb{Q}_m) \approx x^*(\mathbb{P}_n)$ with respect to the stochastic programming problem~\eqref{eq:GeneralTwoStageProblem}, evaluated based on the $\RAE$ in Equation~\eqref{eq:RAE}.

\begin{algorithm}
\setstretch{1.15}
\caption{Scenario Reduction}
\begin{algorithmic}[1]
\State \textbf{inputs} $\mathbb{P}_n$, $c$, $m$
\Statex \textit{Step 1:} \textit{Calculate cost matrix}
\For{$i \in \mathcal{I}$, $j \in \mathcal{I}$}
    \State calculate $c(\xi_i, \xi_j)$ solving problems in Table~\ref{tab:OptProbRequirments}
\EndFor
\Statex \textit{Step 2:} \textit{Determine reduced distribution $\mathbb{Q}_m$}
\State run Algorithm~\cref{alg:ForwardSelection} ($\mathbb{P}_n$, $m$, $c(\xi_i, \xi_j)$)
\Statex receive $\mathbb{Q}_m$
\Statex \textit{Step 3:} \textit{Solve reduced stochastic problem}
\State solve $x^*(\mathbb{Q}_m) \in \arg\min_{x \in \mathcal{X}} ~ z(x,\mathbb{Q}_m)$
\Statex \textit{Step 4:} \textit{Evaluate first-stage decisions}
\State calculate $\RAE(\mathbb{P}_n,\mathbb{Q}_m)$
\State \textbf{return} $\mathbb{Q}_m$, $\RAE(\mathbb{P}_n,\mathbb{Q}_m)$
\end{algorithmic}
\label{alg:ScenarioReduction}
\end{algorithm}

\subsection{Hybrid algorithm}
As we show in the results later, the cost function $c^{\mathrm{Mo}}$~\eqref{eq:c_Mo} is very good at identifying which scenarios are \emph{redundant} at low computational requirement cost, but does not necessarily reveal which \emph{few} scenarios best approximate the optimal first-stage solution of the full problem. Cost function $c^{\mathrm{Pr}}$~\eqref{eq:c_Pr}, on the other hand, is excellent at that job but computationally expensive to evaluate (cf. Table~\ref{tab:requirement_cost}). To overcome their individual drawbacks, we propose the \emph{Hybrid} Algorithm~\cref{alg:Hybrid} targeting the high approximation performance of $c^{\mathrm{Pr}}$ at the low computational cost of $c^{\mathrm{Mo}}$.
In Phase 1, the algorithm uses cost function $c^{\mathrm{Mo}}$ to quickly find a small set of representative scenarios $a \in \mathcal{A} \subset \mathcal{I}$, with $|\mathcal{A}| = r$ and $r \ll n$, that well approximates the initial sample $\mathcal{I}$. Then, in Phase 2, the algorithm applies the cost function $c^{\mathrm{Pr}}$ to further reduce this pre-selection. For that, it is only necessary to solve $r$ single-scenario MIP problems~\eqref{eq:GeneralSingleScenarioDP} and $r^2 - r$ LP problems~\eqref{eq:GeneralSecondStage} for scenarios $a \in \mathcal{A}$, substantially reducing the computational burden. Note that while Theorem~\ref{th:Theorem1} still holds for the pre-selected set of scenarios $\mathcal{A}$, it does not extend to the initial scenario sample~$\mathcal{I}$ when using the Hybrid Algorithm~\cref{alg:Hybrid}. However, we still find that the algorithm often picks the best scenario at $m=1$ and yields good distributions $\mathbb{Q}_m$ in general, as we show in our numerical experiments in the next section.

\providecommand{\algphase}[1]{\Statex\textsc{#1}}
\begin{algorithm}[t]
\caption{Hybrid Scenario Reduction}
\label{alg:Hybrid}
\begin{algorithmic}[1]
\State \textbf{inputs} $\mathbb{P}_n$, $m$, $r$
\algphase{Phase 1: Pre-group sample $\mathbb{P}_n$ with $c^{\mathrm{Mo}}$}
\State run Steps~1--2 of Algorithm~\cref{alg:ScenarioReduction} $(\mathbb{P}_n, c^{\mathrm{Mo}}, r)$
\Statex receive $\mathbb{Q}_r$
\algphase{Phase 2: Reduce $\mathbb{Q}_r$ with $c^{\mathrm{Pr}}$}
\State run Steps~1--2 of Algorithm~\cref{alg:ScenarioReduction} $(\mathbb{Q}_r, c^{\mathrm{Pr}}, m)$
\Statex receive $\mathbb{Q}_m$
\algphase{Continue with Steps 3--4 of Algorithm~\cref{alg:ScenarioReduction}}
\State \textbf{return} $\mathbb{Q}_m$, $\RAE(\mathbb{P}_n,\mathbb{Q}_m)$
\end{algorithmic}
\end{algorithm}

\section{Numerical results}
\label{sec:results}
In the following, we present numerical results obtained by applying the scenario reduction algorithms presented in the last section to a well-studied two-stage stochastic MIP problem, namely, the unit commitment problem, which we present in detail in~\ref{sec:app_SUC}, while varying the cost functions $c$ introduced in Section~\ref{sec:CostFunctionMTP}. The optimization problems and algorithms are implemented in the Julia Programming language v1.12.2 \citep{bezanson2017julia} using JuMP v1.30.1 \citep{Lubin2023}, and the code, together with the case study data, is available on GitHub \citep{github_SRSUC}. All models are solved with Gurobi v.13.0.2 \citep{gurobi} on a single node of the MUSICA HPC equipped with two AMD EPYC 9654 processors using up to 192 cores in total, clocking at a base rate of \qty{2.4}{\giga\hertz} and up to a maximum of \qty{3.7}{\giga\hertz} in turbo mode.

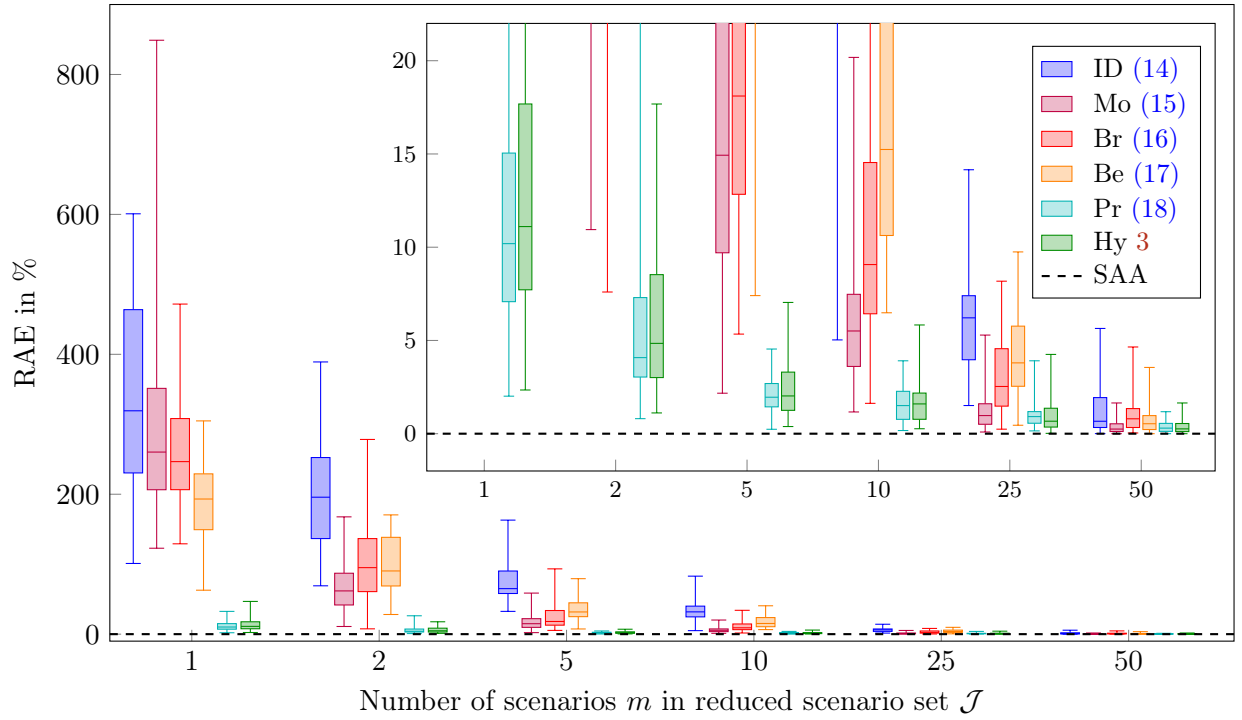
\begin{figure}
    \pgfplotstableread[col sep=comma]{plots/rae_boxplot_data_paper.csv}\datatable
\pgfplotstablegetrowsof{\datatable}
\pgfmathsetmacro{\numrows}{\pgfplotsretval-1}

\colorlet{colIDDSR}{blue}
\colorlet{colMorales}{purple}
\colorlet{colBruninx}{red}
\colorlet{colBertsimas}{orange}
\colorlet{colBertsimasUnstable}{teal!60!cyan}
\colorlet{colHybridMo}{green!55!black}


\begin{tikzpicture}
\begin{axis}[
    name=mainplot, boxplot/draw direction=y,
    xmin=0, xmax=48, ymin=-10, ymax=900,
    xtick={3.5,11.5,19.5,27.5,35.5,43.5},
    xticklabels={1,2,5,10,25,50},
    xlabel={Number of scenarios $m$ in reduced scenario set $\mathcal{J}$},
    ylabel={$\mathrm{RAE}$ in \unit{\percent}},
    width=\linewidth, height=10cm, ymajorgrids=false,
    grid style={gray!30}, xtick pos=bottom, ytick pos=left,
]
\pgfplotsinvokeforeach{0,...,\numrows}{
    \pgfplotstablegetelem{#1}{method}\of\datatable  \let\currentmethod\pgfplotsretval
    \pgfplotstablegetelem{#1}{cluster}\of\datatable \let\currentcluster\pgfplotsretval
    \pgfplotstablegetelem{#1}{min}\of\datatable     \let\currentmin\pgfplotsretval
    \pgfplotstablegetelem{#1}{q1}\of\datatable      \let\currentqone\pgfplotsretval
    \pgfplotstablegetelem{#1}{median}\of\datatable  \let\currentmedian\pgfplotsretval
    \pgfplotstablegetelem{#1}{q3}\of\datatable      \let\currentqthree\pgfplotsretval
    \pgfplotstablegetelem{#1}{max}\of\datatable     \let\currentmax\pgfplotsretval
    \ifnum\pdfstrcmp{\currentmethod}{IDDSR}=0             \def\methodidx{0}\def\methodcolor{colIDDSR}\fi
    \ifnum\pdfstrcmp{\currentmethod}{Morales}=0           \def\methodidx{1}\def\methodcolor{colMorales}\fi
    \ifnum\pdfstrcmp{\currentmethod}{Bruninx}=0           \def\methodidx{2}\def\methodcolor{colBruninx}\fi
    \ifnum\pdfstrcmp{\currentmethod}{Bertsimas}=0         \def\methodidx{3}\def\methodcolor{colBertsimas}\fi
    \ifnum\pdfstrcmp{\currentmethod}{BertsimasUnstable}=0 \def\methodidx{4}\def\methodcolor{colBertsimasUnstable}\fi
    \ifnum\pdfstrcmp{\currentmethod}{HybridMo}=0          \def\methodidx{5}\def\methodcolor{colHybridMo}\fi
    \ifnum\currentcluster=1  \def\clusteridx{0}\fi
    \ifnum\currentcluster=2  \def\clusteridx{1}\fi
    \ifnum\currentcluster=5  \def\clusteridx{2}\fi
    \ifnum\currentcluster=10 \def\clusteridx{3}\fi
    \ifnum\currentcluster=25 \def\clusteridx{4}\fi
    \ifnum\currentcluster=50 \def\clusteridx{5}\fi
    \pgfmathsetmacro{\xpos}{1 + \clusteridx*8 + \methodidx}
    \edef\temp{\noexpand\addplot+[boxplot prepared={lower whisker=\currentmin,
        lower quartile=\currentqone, median=\currentmedian, upper quartile=\currentqthree,
        upper whisker=\currentmax, draw position=\xpos},
        fill=\methodcolor!30, draw=\methodcolor, solid, forget plot,] coordinates {};}
    \temp
}
\addplot[black, thick, dashed, domain=0:48, samples=2, mark=none, forget plot] {0};
\end{axis}
\begin{axis}[
    name=insetplot, at={(mainplot.north east)}, anchor=north east,
    xshift=-0.25cm, yshift=-0.25cm, boxplot/draw direction=y,
    xmin=0, xmax=48, ymin=-2, ymax=22,
    xtick={3.5,11.5,19.5,27.5,35.5,43.5},
    xticklabels={1,2,5,10,25,50},
    xlabel={}, ylabel={}, width=0.73\linewidth, height=7.5cm,
    ymajorgrids=false, grid style={gray!30}, axis background/.style={fill=white},
    tick label style={font=\footnotesize}, xtick pos=bottom, ytick pos=left,
    legend style={at={(0.875,0.95)}, anchor=north, legend columns=1,
        cells={anchor=west}, font=\small, fill=white, fill opacity=0.9,
        draw opacity=1, text opacity=1},
    legend image code/.code={\draw[#1, fill] (0cm,-0.1cm) rectangle (0.4cm,0.1cm);},
]
\addlegendimage{fill=colIDDSR!30, draw=colIDDSR}\addlegendentry{ID~\eqref{eq:c_ID}}
\addlegendimage{fill=colMorales!30, draw=colMorales}\addlegendentry{Mo~\eqref{eq:c_Mo}}
\addlegendimage{fill=colBruninx!30, draw=colBruninx}\addlegendentry{Br~\eqref{eq:c_Br}}
\addlegendimage{fill=colBertsimas!30, draw=colBertsimas}\addlegendentry{Be~\eqref{eq:c_Be}}
\addlegendimage{fill=colBertsimasUnstable!30, draw=colBertsimasUnstable}\addlegendentry{Pr~\eqref{eq:c_Pr}}
\addlegendimage{fill=colHybridMo!30, draw=colHybridMo}\addlegendentry{Hy~\cref{alg:Hybrid}} 
\addlegendimage{line legend, black, thick, dashed}\addlegendentry{SAA}
\pgfplotsinvokeforeach{0,...,\numrows}{
    \pgfplotstablegetelem{#1}{method}\of\datatable  \let\currentmethod\pgfplotsretval
    \pgfplotstablegetelem{#1}{cluster}\of\datatable \let\currentcluster\pgfplotsretval
    \pgfplotstablegetelem{#1}{min}\of\datatable     \let\currentmin\pgfplotsretval
    \pgfplotstablegetelem{#1}{q1}\of\datatable      \let\currentqone\pgfplotsretval
    \pgfplotstablegetelem{#1}{median}\of\datatable  \let\currentmedian\pgfplotsretval
    \pgfplotstablegetelem{#1}{q3}\of\datatable      \let\currentqthree\pgfplotsretval
    \pgfplotstablegetelem{#1}{max}\of\datatable     \let\currentmax\pgfplotsretval
    \ifnum\pdfstrcmp{\currentmethod}{IDDSR}=0             \def\methodidx{0}\def\methodcolor{colIDDSR}\fi
    \ifnum\pdfstrcmp{\currentmethod}{Morales}=0           \def\methodidx{1}\def\methodcolor{colMorales}\fi
    \ifnum\pdfstrcmp{\currentmethod}{Bruninx}=0           \def\methodidx{2}\def\methodcolor{colBruninx}\fi
    \ifnum\pdfstrcmp{\currentmethod}{Bertsimas}=0         \def\methodidx{3}\def\methodcolor{colBertsimas}\fi
    \ifnum\pdfstrcmp{\currentmethod}{BertsimasUnstable}=0 \def\methodidx{4}\def\methodcolor{colBertsimasUnstable}\fi
    \ifnum\pdfstrcmp{\currentmethod}{HybridMo}=0          \def\methodidx{5}\def\methodcolor{colHybridMo}\fi
    \ifnum\currentcluster=1  \def\clusteridx{0}\fi
    \ifnum\currentcluster=2  \def\clusteridx{1}\fi
    \ifnum\currentcluster=5  \def\clusteridx{2}\fi
    \ifnum\currentcluster=10 \def\clusteridx{3}\fi
    \ifnum\currentcluster=25 \def\clusteridx{4}\fi
    \ifnum\currentcluster=50 \def\clusteridx{5}\fi
    \pgfmathsetmacro{\xpos}{1 + \clusteridx*8 + \methodidx}
    \pgfmathsetmacro{\clippedmin}{max(\currentmin,-2)}
    \pgfmathsetmacro{\clippedmax}{min(\currentmax,100)}
    \pgfmathsetmacro{\clippedqone}{max(min(\currentqone,100),-2)}
    \pgfmathsetmacro{\clippedmedian}{max(min(\currentmedian,100),-2)}
    \pgfmathsetmacro{\clippedqthree}{max(min(\currentqthree,100),-2)}
    \edef\temp{\noexpand\addplot+[boxplot prepared={lower whisker=\clippedmin,
        lower quartile=\clippedqone, median=\clippedmedian, upper quartile=\clippedqthree,
        upper whisker=\clippedmax, draw position=\xpos},
        fill=\methodcolor!30, draw=\methodcolor, solid, forget plot,] coordinates {};}
    \temp
}
\addplot[black, thick, dashed, domain=0:48, samples=2, mark=none, forget plot] {0};
\end{axis}
\end{tikzpicture}
    \caption{$\RAE$ for the cost functions $c \in \mathcal{C}$ using the Scenario Reduction Algorithm~\cref{alg:ScenarioReduction} and the Hybrid Algorithm~\cref{alg:Hybrid} with $r=50$ depending on the size $m$ of the reduced scenario set $\mathcal{J}$ for the modified IEEE RTS 24-bus system described in Section~\ref{sec:CaseStudy_24Bus}. Box plots represent variation across 50 randomly sampled distributions $\mathbb{P}_n$ of size $n = 200$, drawn from the probabilistic forecast in~\citet{pinson2013}. The dashed line indicates the SAA solution.}
    \label{fig:IEEE24_quality_draws}
\end{figure}

\subsection{Small case study: IEEE RTS 24-bus system}
\label{sec:CaseStudy_24Bus}

The first case study is a modified version of the IEEE RTS 24-bus system presented in \citet{Ordoudis2016} with 32 dispatchable generating units to be committed \citep{Subcommittee1979}. We add a generator-specific fixed cost term $C^{\mathrm{fix}}$ and modify the linear production costs $C^{\mathrm{L}}$ based on the values presented in the Power Grid Lib~\citep{Babaeinejadsarookolaee2019}. Finally, we slightly perturb all cost coefficients to ensure that each generator has a unique cost function and reduce the potential for multiple optimal solutions.
The power system hosts six intermittent wind power generators at nodes \{3,5,7,16,21,23\}, each with an installed capacity of \SI{350}{\mega\watt}. We assume that the uncertain wind power production can be approximated by the spatially and temporally correlated probabilistic forecast presented in \cite{pinson2013}, which is based on real measurements of wind farms in Denmark.
We apply a uniform load shedding cost $C^{\mathrm{Sh}}$ of \qty{1500}{\$\per\mega\watt\hour} and assume that wind power can be curtailed at no cost.
Finally, we use Gurobi's standard solver settings unless otherwise indicated. In particular, we use a MIP gap of \qty{0.01}{\percent} for the small case study, enabling an unbiased comparison of the quality of the approximation. We leave a detailed analysis of the computational expenses for the large case study in Section~\ref{sec:large_case}

We assess the impact of using the different cost functions $c$ discussed in Section~\ref{sec:CostFunctionMTP} by constructing 50 independent original distributions $\mathbb{P}_n$, each obtained by sampling $n=200$ atoms from the available probabilistic forecast. Figure~\ref{fig:IEEE24_quality_draws} shows the distribution of the $\RAE$ over the draws for different sizes $m$ of the reduced set of scenarios $\mathcal{J}$. The dashed line denotes an $\RAE$ of \qty{0}{\percent} obtained when solving the two-stage stochastic MIP~\eqref{eq:GeneralTwoStageProblem} under distribution $\mathbb{P}_n$, equaling the sample average approximation (SAA).

We first examine the approximation quality of the cost functions discussed in Section~\ref{sec:CostFunctionMTP} before analyzing the results of the Hybrid Algorithm~\cref{alg:Hybrid}.
The first thing we would like to point out is that for a single scenario, i.e., $m=1$, all cost functions except $c^{\mathrm{Pr}}$ yield quite inaccurate reduced sets of scenarios with low approximation quality, as measured by the $\RAE$. Additionally, the approximation quality varies widely depending on the specific sample drawn. For the cost function $c^{\mathrm{Pr}}$, on the other hand, we can observe two things: (1) The numerical simulations confirm the result in Theorem~\ref{th:Theorem1}, i.e., that always the best possible scenario in terms of $\RAE$ is selected, leading to a comparably accurate approximation of the first-stage decisions even for $m=1$, and (2) that because it is guaranteed to select this best scenario independently of the original sample distribution $\mathbb{P}_n$ drawn, it delivers much more robust results when varying this distribution. In fact, it is often possible to achieve an $\RAE$ of around \qty{10}{\percent} with only a single scenario. While Theorem~\ref{th:Theorem1} only considers the case of $m=1$, we can see that the distributions for higher numbers of scenarios produced by the cost function $c^{\mathrm{Pr}}$ are of very good quality when compared to the reduced scenario sets delivered by the rest of the methods.

As shown in the inset plot of Figure~\ref{fig:IEEE24_quality_draws}, all cost functions start delivering distributions with good approximation quality from 25 scenarios onward for most draws. Depending on the desired level of accuracy, however, we find that when using $c^{\mathrm{Pr}}$, the optimal first-stage decisions related to the original distribution $\mathbb{P}_n$ can be approximated with an $\RAE$ of around \qty{2}{\percent} with only 5 scenarios. Notably, we also find that $c^{\mathrm{Mo}}$ for $m>1$ delivers more accurate results than cost functions $c^{\mathrm{Br}}$ and $c^{\mathrm{Be}}$ and is much easier to evaluate (cf. Table~\ref{tab:requirement_cost}).

Overall, the proposed Hybrid Algorithm~\cref{alg:Hybrid} achieves an approximation quality very similar to that of the cost function~$c^{\mathrm{Pr}}$, despite its pre-selection step using the cost function~$c^{\mathrm{Mo}}$ and $ r=50$ scenarios. Most notably, this also holds when the reduced scenario set $\mathcal{J}$ is small. In particular, it indicates that although Theorem~\ref{th:Theorem1} is not satisfied for the original sample of scenarios $\mathcal{I}$, it is sufficient that it holds for the pre-selected set of scenarios $\mathcal{R}$, as the cost function~$c^{\mathrm{Mo}}$ is performing well at that step. While not compromising solution quality, this pre-selection drastically reduces the number of optimization problems to be solved (cf. Table~\ref{tab:OptProbRequirments}) and the corresponding computational expenses. In the following section, we analyze those expenses using a realistic large-scale case study.

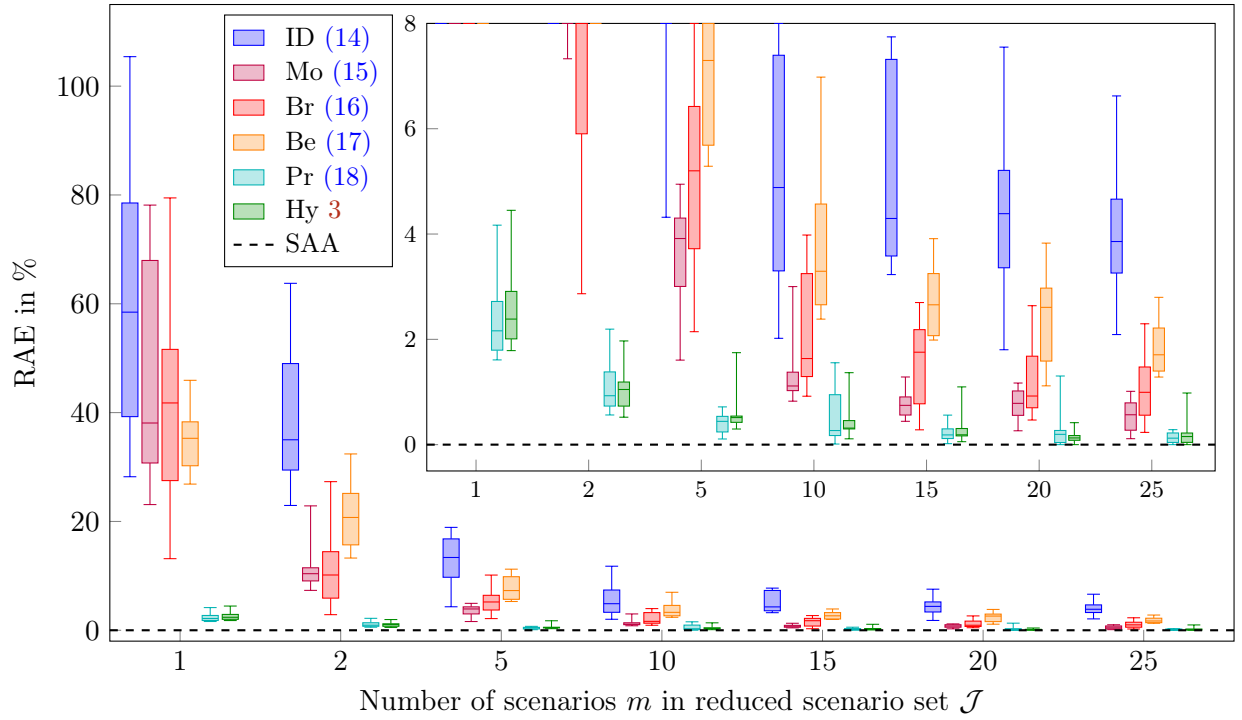
\begin{figure}
    \pgfplotstableread[col sep=comma]{plots/rae_boxplot_data_ieee300.csv}\datatable
\pgfplotstablegetrowsof{\datatable}
\pgfmathsetmacro{\numrows}{\pgfplotsretval-1}

\colorlet{colIDDSR}{blue}
\colorlet{colMorales}{purple}
\colorlet{colBruninx}{red}
\colorlet{colBertsimas}{orange}
\colorlet{colBertsimasUnstable}{teal!60!cyan}
\colorlet{colHybridMo}{green!55!black}


\newcommand{\mapmethod}{%
  \def\mi{-1}\def\mc{colIDDSR}
  \ifnum\pdfstrcmp{\cm}{IDDSR}=0             \def\mi{0}\def\mc{colIDDSR}\fi
  \ifnum\pdfstrcmp{\cm}{Morales}=0           \def\mi{1}\def\mc{colMorales}\fi
  \ifnum\pdfstrcmp{\cm}{Bruninx}=0           \def\mi{2}\def\mc{colBruninx}\fi
  \ifnum\pdfstrcmp{\cm}{Bertsimas}=0         \def\mi{3}\def\mc{colBertsimas}\fi
  \ifnum\pdfstrcmp{\cm}{BertsimasUnstable}=0 \def\mi{4}\def\mc{colBertsimasUnstable}\fi
  \ifnum\pdfstrcmp{\cm}{Hybrid-Mo}=0         \def\mi{5}\def\mc{colHybridMo}\fi}
\newcommand{\mapcluster}{%
  \def\ki{0}%
  \ifnum\cl=1 \def\ki{0}\fi  \ifnum\cl=2 \def\ki{1}\fi  \ifnum\cl=5 \def\ki{2}\fi
  \ifnum\cl=10 \def\ki{3}\fi \ifnum\cl=15 \def\ki{4}\fi \ifnum\cl=20 \def\ki{5}\fi
  \ifnum\cl=25 \def\ki{6}\fi}

\begin{tikzpicture}
\begin{axis}[name=mainplot, boxplot/draw direction=y,
    xmin=0, xmax=56, ymin=-2, ymax=115,
    xtick={3.5,11.5,19.5,27.5,35.5,43.5,51.5}, xticklabels={1,2,5,10,15,20,25},
    xlabel={Number of scenarios $m$ in reduced scenario set $\mathcal{J}$},
    ylabel={$\mathrm{RAE}$ in \unit{\percent}},
    width=\linewidth, height=10cm, ymajorgrids=false, xtick pos=bottom, ytick pos=left]
\pgfplotsinvokeforeach{0,...,\numrows}{
    \pgfplotstablegetelem{#1}{method}\of\datatable  \let\cm\pgfplotsretval
    \pgfplotstablegetelem{#1}{cluster}\of\datatable \let\cl\pgfplotsretval
    \pgfplotstablegetelem{#1}{min}\of\datatable     \let\vmin\pgfplotsretval
    \pgfplotstablegetelem{#1}{q1}\of\datatable      \let\vqo\pgfplotsretval
    \pgfplotstablegetelem{#1}{median}\of\datatable  \let\vmd\pgfplotsretval
    \pgfplotstablegetelem{#1}{q3}\of\datatable      \let\vqt\pgfplotsretval
    \pgfplotstablegetelem{#1}{max}\of\datatable     \let\vmax\pgfplotsretval
    \mapmethod \mapcluster
    \ifnum\mi>-1
    \pgfmathsetmacro{\xp}{1 + \ki*8 + \mi}
    \pgfmathsetmacro{\bmin}{max(\vmin,-2)}\pgfmathsetmacro{\bmax}{min(\vmax,114)}
    \pgfmathsetmacro{\bqo}{max(min(\vqo,114),-2)}\pgfmathsetmacro{\bmd}{max(min(\vmd,114),-2)}
    \pgfmathsetmacro{\bqt}{max(min(\vqt,114),-2)}
    \edef\tmp{\noexpand\addplot+[boxplot prepared={lower whisker=\bmin, lower quartile=\bqo,
      median=\bmd, upper quartile=\bqt, upper whisker=\bmax, draw position=\xp},
      fill=\mc!30, draw=\mc, solid, forget plot,] coordinates {};}\tmp
    \fi
}
\addplot[black, thick, dashed, domain=0:56, samples=2, mark=none, forget plot] {0};
\end{axis}
\begin{axis}[name=insetplot, at={(mainplot.north east)}, anchor=north east,
    xshift=-0.25cm, yshift=-0.25cm, boxplot/draw direction=y,
    xmin=0, xmax=56, ymin=-0.5, ymax=8,
    xtick={3.5,11.5,19.5,27.5,35.5,43.5,51.5}, xticklabels={1,2,5,10,15,20,25},
    width=0.73\linewidth, height=7.5cm, ymajorgrids=false, axis background/.style={fill=white},
    tick label style={font=\footnotesize}, xtick pos=bottom, ytick pos=left,
    legend style={at={(-0.15,1.02)}, anchor=north, legend columns=1, cells={anchor=west},
      font=\small, fill=white, fill opacity=0.9, draw opacity=1, text opacity=1},
    legend image code/.code={\draw[#1, fill] (0cm,-0.1cm) rectangle (0.4cm,0.1cm);}]
    \addlegendimage{fill=colIDDSR!30, draw=colIDDSR}\addlegendentry{ID~\eqref{eq:c_ID}}
    \addlegendimage{fill=colMorales!30, draw=colMorales}\addlegendentry{Mo~\eqref{eq:c_Mo}}
    \addlegendimage{fill=colBruninx!30, draw=colBruninx}\addlegendentry{Br~\eqref{eq:c_Br}}
    \addlegendimage{fill=colBertsimas!30, draw=colBertsimas}\addlegendentry{Be~\eqref{eq:c_Be}}
    \addlegendimage{fill=colBertsimasUnstable!30, draw=colBertsimasUnstable}\addlegendentry{Pr~\eqref{eq:c_Pr}}
    \addlegendimage{fill=colHybridMo!30, draw=colHybridMo}\addlegendentry{Hy~\cref{alg:Hybrid}} 
    \addlegendimage{line legend, black, thick, dashed}\addlegendentry{SAA}
\pgfplotsinvokeforeach{0,...,\numrows}{
    \pgfplotstablegetelem{#1}{method}\of\datatable  \let\cm\pgfplotsretval
    \pgfplotstablegetelem{#1}{cluster}\of\datatable \let\cl\pgfplotsretval
    \pgfplotstablegetelem{#1}{min}\of\datatable     \let\vmin\pgfplotsretval
    \pgfplotstablegetelem{#1}{q1}\of\datatable      \let\vqo\pgfplotsretval
    \pgfplotstablegetelem{#1}{median}\of\datatable  \let\vmd\pgfplotsretval
    \pgfplotstablegetelem{#1}{q3}\of\datatable      \let\vqt\pgfplotsretval
    \pgfplotstablegetelem{#1}{max}\of\datatable     \let\vmax\pgfplotsretval
    \mapmethod \mapcluster
    \ifnum\mi>-1
    \pgfmathsetmacro{\xp}{1 + \ki*8 + \mi}
    \pgfmathsetmacro{\bmin}{max(\vmin,-0.5)}\pgfmathsetmacro{\bmax}{min(\vmax,8)}
    \pgfmathsetmacro{\bqo}{max(min(\vqo,8),-0.5)}\pgfmathsetmacro{\bmd}{max(min(\vmd,8),-0.5)}
    \pgfmathsetmacro{\bqt}{max(min(\vqt,8),-0.5)}
    \edef\tmp{\noexpand\addplot+[boxplot prepared={lower whisker=\bmin, lower quartile=\bqo,
      median=\bmd, upper quartile=\bqt, upper whisker=\bmax, draw position=\xp},
      fill=\mc!30, draw=\mc, solid, forget plot,] coordinates {};}\tmp
    \fi
}
\addplot[black, thick, dashed, domain=0:56, samples=2, mark=none, forget plot] {0};
\end{axis}
\end{tikzpicture}
    \caption{$\RAE$ for the cost functions $c \in \mathcal{C}$ using the Scenario Reduction Algorithm~\cref{alg:ScenarioReduction} and the Hybrid Algorithm~\cref{alg:Hybrid} with $r=50$ depending on the size $m$ of the reduced scenario set $\mathcal{J}$ for the modified IEEE 300-bus system described in Section~\ref{sec:CaseStudy_24Bus}. Box plots represent variation across 10 randomly sampled distributions $\mathbb{P}_n$ of size $n = 1000$, drawn from the probabilistic forecast in~\citet{pinson2013}. The dashed line indicates the SAA solution.}
    \label{fig:IEEE300_quality_draws}
\end{figure}

\subsection{Large case study: IEEE 300-bus system}
\label{sec:large_case}

The large case study is a modified version of the IEEE 300-bus system presented in the Power Grid Lib~ \citep{Babaeinejadsarookolaee2019}, where we enrich the dispatchable generators with the technical characteristics of the RTS-GMLC~\citep{Barrows2020,rtsgmlc_repo} to represent unit commitment and split them into 170 individual committable units with \SI{35}{\giga\watt} of capacity in total. We further add 17 intermittent wind power generators with \SIrange{600}{3000}{\mega\watt} installed capacity, totaling \SI{26.6}{\giga\watt} and again utilize the probabilistic forecast from \cite{pinson2013}. The uniform load shedding costs $C^{\mathrm{Sh}}$ are assumed to be \qty{3000}{\$\per\mega\watt\hour}.
Due to the size of the problems, we increase the MIP gap to \qty{1}{\percent} for all MIPs solved.

\subsubsection{Quality of approximation}
Figure~\ref{fig:IEEE300_quality_draws} shows a comparison of the RAE for different sizes $m$ of the reduced scenario set $\mathcal{J}$ across 10 random draws with $n=1000$ scenarios. For the Hybrid Algorithm~\cref{alg:Hybrid}, we use $r=50$ pre-selected scenarios in Phase 1, which we find to perform well and robustly across draws. While the results generally confirm the findings from the small case study, the overall RAE and its deviations are slightly smaller due to the increased flexibility of a larger power system in balancing fluctuations. Two things are, however, again worth noting for the cost function $c^{\mathrm{Pr}}$: First, the scenario sets found exhibit less variance in approximation quality over different samples drawn. Secondly, the first scenario drawn when using the cost function $c^{\mathrm{Pr}}$ leads to a mean RAE of only around \qty{2.4}{\percent}, leading up to \qty{0.4}{\percent} with only $m=5$ scenarios in the reduced set. We also point out that the Hybrid Algorithm~\cref{alg:Hybrid} achieves an approximation quality essentially indistinguishable from that of $c^{\mathrm{Pr}}$, despite requiring far less computation to evaluate, which we discuss in detail next.

\subsubsection{Computational expenses}
\label{sec:large_case_comp_exp}
We have tested several configurations to minimize wall-clock time on the hardware required to solve the optimization problems for each cost function shown in Table~\ref{tab:OptProbRequirments} and present the best ones in~\ref{sec:app_configs}.

Table~\ref{tab:requirement_cost} shows the mean and standard deviation, taken over different draws, of Gurobi's work units (left) and the wall-clock time (right) required to solve the optimization problems in Table~\ref{tab:OptProbRequirments} for the different cost functions and the Hybrid Algorithm~\cref{alg:Hybrid}.

When looking at the work units on the left, one can clearly observe that solving $n$ MIP and $n^2-n$ LP problems for $n = 1000$ becomes quite expensive. It is, however, possible to calculate all those problems in parallel, even across machines, with barely any overhead. When running the problems in parallel on 192 cores, the wall-clock times can be substantially reduced, as shown on the right side of Table~\ref{tab:requirement_cost}.
Notably, the work units and wall-clock times used to calculate the requirements vary little across the originally drawn samples, as indicated by the small standard deviation. 

One thing that becomes clear immediately is that the cost functions $c^{\mathrm{Be}}$ and $c^{\mathrm{Pr}}$ are much more expensive to evaluate than the others, even with massive parallelization. For $n=1000$ scenarios in particular, solving the second-stage LP problems becomes computationally expensive. This problem is overcome by our proposed Hybrid Algorithm~\eqref{alg:Hybrid}, which drastically limits the number of MIP and LP problems by performing the pre-selection with cost function $c^{\mathrm{Mo}}$, reducing the set of scenarios to be evaluated with $c^{\mathrm{Pr}}$ from $n=1000$ to $r=50$ in Phase 2. Accordingly, it requires 18 times less wall-clock time and 66 times less work (as measured by Gurobi) than using only the proposed cost function $c^{\mathrm{Pr}}$.

    \begin{table}
  \centering
  \footnotesize\setlength{\tabcolsep}{3pt}
   \begin{threeparttable}
  \caption{Mean and standard deviation (in brackets) of Gurobi's work units and wall-clock time required for solving the optimization problems in \textit{Step~1} of Algorithm~\cref{alg:ScenarioReduction} and the Hybrid Algorithm~\cref{alg:Hybrid} with $r=50$ for the modified IEEE 300-bus system described in Section~\ref{sec:large_case} with $n=1000$ scenarios over 10 randomly sampled distributions $\mathbb{P}_n$.} 
  \label{tab:requirement_cost}
  \begin{tabular}{lrrrrrr}
  \toprule
  & \multicolumn{3}{c}{Work units} & \multicolumn{3}{c}{Wall-clock in \unit{\second}} \\
  \cmidrule(lr){2-4}\cmidrule(lr){5-7}
  Cost func. $c$  & MILP & LP & Total & MILP & LP & Total \\
  \midrule
  ID~\eqref{eq:c_ID} & -- & -- & -- & -- & -- & -- \\
  Mo~\eqref{eq:c_Mo} & 56.0 (28.2) & 205.0 (0.9) & 260.9 (28.7) & 33.5 (14.5) & 22.6 (5.2) & 56.1 (16.9) \\
  Br~\eqref{eq:c_Br} & 46\,391.2 (243.4) & -- & 46\,391.2 (243.4) & 595.6 (58.2) & -- & 595.6 (58.2) \\
  Be~\eqref{eq:c_Be} & 46\,391.2 (243.4) & 153\,119.0 (917.9) & 199\,510.2 (1\,060.5) & 595.6 (58.2) & 2\,093.9 (131.7) & 2\,689.5 (185.3) \\
  Pr~\eqref{eq:c_Pr} & 46\,391.2 (243.4) & 153\,119.0 (917.9) & 199\,510.2 (1\,060.5) & 595.6 (58.2) & 2\,093.9 (131.7) & 2\,689.5 (185.3) \\
  Hy~\cref{alg:Hybrid}\tnote{a} & 2\,396.5 (91.8) & 624.1 (15.9) & 3\,020.6 (84.9) & 95.6 (15.5) & 50.6 (5.0) & 146.2 (18.1) \\
  \bottomrule
  \end{tabular}
  \begin{tablenotes}\footnotesize
  \item[a] Columns show aggregated values for Phase~1 and Phase~2.
  \end{tablenotes}
  \end{threeparttable}
\end{table}

To fairly compare the different scenario reduction methods in terms of computational complexity, we further compare the combined effort for performing the scenario reduction and solving the reduced two-stage stochastic MIP problem. For that, we also utilize warm-starting of the MIPs as elaborated in~\ref{sec:app_configs}.
Figure~\ref{fig:agg_time_comp} compares these combined wall-clock times (mean over draws) for different sizes of the reduced scenario set $\mathcal{J}$. 
While wall-clock times for the reduced two-stage stochastic MIPs are similar across scenario reduction methods for a given $m$, the combined times differ substantially. 
One can immediately notice that for measures $c^{\mathrm{Br}}, c^{\mathrm{Be}},$ and $c^{\mathrm{Pr}}$, performing the scenario reduction is way more expensive than solving the corresponding reduced two-stage stochastic MIPs for $m<15$. In fact, it takes until $m=20$ for methods $c^{\mathrm{ID}}, c^{\mathrm{Mo}}$, and the Hybrid Algorithm~\cref{alg:Hybrid} to match the total wall-clock time of methods $c^{\mathrm{Be}}$ and $c^{\mathrm{Pr}}$ for $m=1$. This highlights that while the proposed cost function $c^{\mathrm{Pr}}$ yields very good reduced distributions, it takes substantially more time than solving the reduced two-stage MIPs, potentially rendering it impractical. This drawback is overcome by our proposed Hybrid Algorithm~\cref{alg:Hybrid}, which only slightly increases combined wall-clock times compared to the cost functions $c^{\mathrm{ID}}$ and $c^{\mathrm{Mo}}$. Hence, we will take a more detailed look at individual performance by jointly considering computational expenses and solution quality.

\subsubsection{Combined evaluation}

Taking the quality of the solution shown in Figure~\ref{fig:IEEE300_quality_draws} and the analysis of the computational expenses jointly into account, we highlight a few points:
1) Function $c^{\mathrm{Pr}}$ outperforms $c^{\mathrm{Be}}$ on the quality of approximation, while having the same computational expenses for evaluating the distance and faster solution times of the reduced two-stage MIPs. 2) When comparing the combined wall-clock times, across a number of scenarios to achieve a given solution quality, e.g., an RAE of \qty{1}{\percent}, cost function $c^{\mathrm{Mo}}$ for $m=15$ achieves a better solution in less time compared to $c^{\mathrm{Pr}}$ with $m=5$ scenarios, due to the long times it takes to calculate the requirements -- even including heavy parallelization. 3) The Hybrid Algorithm~\cref{alg:Hybrid} achieves the same approximation quality as $c^{\mathrm{Pr}}$ with way less computational expenses. As a consequence, it also outperforms the cost function $c^{\mathrm{Mo}}$. 
The proposed Hybrid Algorithm~\cref{alg:Hybrid} therefore provides an accurate and fast method for scenario reduction in two-stage stochastic MIPs.

\begin{figure}
\pgfplotstableread[col sep=comma]{plots/cost_stack_ow_ieee300.csv}\datatable
\pgfplotstablegetrowsof{\datatable}
\pgfmathsetmacro{\numrows}{\pgfplotsretval-1}

\colorlet{colIDDSR}{blue}
\colorlet{colMorales}{purple}
\colorlet{colBruninx}{red}
\colorlet{colBertsimas}{orange}
\colorlet{colBertsimasUnstable}{teal!60!cyan}
\colorlet{colHybridMo}{green!55!black}


\begin{tikzpicture}
\begin{axis}[
    ymin=0, ymax=8500, xmin=0, xmax=56,
    ytick={0,2000,4000,6000,8000},
    xtick={3.5,11.5,19.5,27.5,35.5,43.5,51.5}, xticklabels={1,2,5,10,15,20,25},
    xlabel={Number of scenarios $m$ in reduced scenario set $\mathcal{J}$},
    ylabel={Wall-clock time [\unit{\second}]},
    width=\linewidth, height=10cm, xtick pos=bottom, ytick pos=left,
    ymajorgrids=false, grid style={gray!25},
    legend style={at={(0.02,0.97)}, anchor=north west, legend columns=1, cells={anchor=west},
                  font=\small, fill=white, fill opacity=0.9, draw opacity=1, text opacity=1},
    legend image code/.code={\draw[#1, fill] (0cm,-0.1cm) rectangle (0.35cm,0.1cm);},
]
    \addlegendimage{fill=colIDDSR, draw=colIDDSR}\addlegendentry{ID~\eqref{eq:c_ID}}
    \addlegendimage{fill=colMorales, draw=colMorales}\addlegendentry{Mo~\eqref{eq:c_Mo}}
    \addlegendimage{fill=colBruninx, draw=colBruninx}\addlegendentry{Br~\eqref{eq:c_Br}}
    \addlegendimage{fill=colBertsimas, draw=colBertsimas}\addlegendentry{Be~\eqref{eq:c_Be}}
    \addlegendimage{fill=colBertsimasUnstable, draw=colBertsimasUnstable}\addlegendentry{Pr~\eqref{eq:c_Pr}}
    \addlegendimage{fill=colHybridMo, draw=colHybridMo}\addlegendentry{Hy~\cref{alg:Hybrid}} 
    \addlegendimage{fill=black,draw=black}                               \addlegendentry{Compute requirements}
\addlegendimage{fill=black!30,draw=black}                            \addlegendentry{Solve stochastic MIP~\eqref{eq:TwoStageOptSolQ}}

\pgfplotsinvokeforeach{0,...,\numrows}{
    \pgfplotstablegetelem{#1}{method}\of\datatable    \let\m\pgfplotsretval
    \pgfplotstablegetelem{#1}{cluster}\of\datatable   \let\c\pgfplotsretval
    \pgfplotstablegetelem{#1}{solve}\of\datatable     \let\sv\pgfplotsretval
    \pgfplotstablegetelem{#1}{total}\of\datatable     \let\tot\pgfplotsretval
    \pgfplotstablegetelem{#1}{reqs_lo}\of\datatable   \let\rlo\pgfplotsretval
    \pgfplotstablegetelem{#1}{reqs_hi}\of\datatable   \let\rhi\pgfplotsretval
    \pgfplotstablegetelem{#1}{solve_lo}\of\datatable  \let\slo\pgfplotsretval
    \pgfplotstablegetelem{#1}{solve_hi}\of\datatable  \let\shi\pgfplotsretval
    \ifnum\pdfstrcmp{\m}{IDDSR}=0             \def\mi{0}\def\mc{colIDDSR}\fi
    \ifnum\pdfstrcmp{\m}{Morales}=0           \def\mi{1}\def\mc{colMorales}\fi
    \ifnum\pdfstrcmp{\m}{Bruninx}=0           \def\mi{2}\def\mc{colBruninx}\fi
    \ifnum\pdfstrcmp{\m}{Bertsimas}=0         \def\mi{3}\def\mc{colBertsimas}\fi
    \ifnum\pdfstrcmp{\m}{BertsimasUnstable}=0 \def\mi{4}\def\mc{colBertsimasUnstable}\fi
    \ifnum\pdfstrcmp{\m}{Hybrid-Mo}=0         \def\mi{5}\def\mc{colHybridMo}\fi
    \def\ci{0}%
    \ifnum\c=1 \def\ci{0}\fi  \ifnum\c=2 \def\ci{1}\fi  \ifnum\c=5 \def\ci{2}\fi
    \ifnum\c=10 \def\ci{3}\fi \ifnum\c=15 \def\ci{4}\fi \ifnum\c=20 \def\ci{5}\fi \ifnum\c=25 \def\ci{6}\fi
    \pgfmathsetmacro{\xp}{1 + \ci*8 + \mi}
    \pgfmathsetmacro{\xl}{\xp-0.45}\pgfmathsetmacro{\xr}{\xp+0.45}
    \edef\tempa{\noexpand\draw[draw=\mc, fill=\mc!30] (axis cs:\xl,0) rectangle (axis cs:\xr,\sv);}\tempa
    \edef\tempb{\noexpand\draw[draw=\mc, fill=\mc] (axis cs:\xl,\sv) rectangle (axis cs:\xr,\tot);}\tempb
    \pgfmathsetmacro{\cw}{0.16}\pgfmathsetmacro{\xa}{\xp-\cw}\pgfmathsetmacro{\xb}{\xp+\cw}
    \pgfmathsetmacro{\rloy}{\sv+\rlo}\pgfmathsetmacro{\rhiy}{\sv+\rhi}
    \edef\temps{\noexpand\draw[black, line width=0.4pt]
        (axis cs:\xp,\slo) -- (axis cs:\xp,\shi)
        (axis cs:\xa,\slo) -- (axis cs:\xb,\slo) (axis cs:\xa,\shi) -- (axis cs:\xb,\shi);}\temps
}
\end{axis}
\end{tikzpicture}
  \caption{Wall-clock times in \protect\si{\second} (mean over draws) required for solving the optimization problems in \textit{Step~1} of Algorithm~\protect\ref{alg:ScenarioReduction} and Phases 1 and 2 of Hybrid Algorithm~\protect\cref{alg:Hybrid} and solving the reduced two-stage stochastic MIP~\protect\eqref{eq:TwoStageOptSolQ} for the modified IEEE 300-bus system described in Section~\protect\ref{sec:large_case} with $n=1000$ scenarios. Whiskers indicate variation over draws of wall-clock time for computing the requirements.}
  \label{fig:agg_time_comp}
\end{figure}
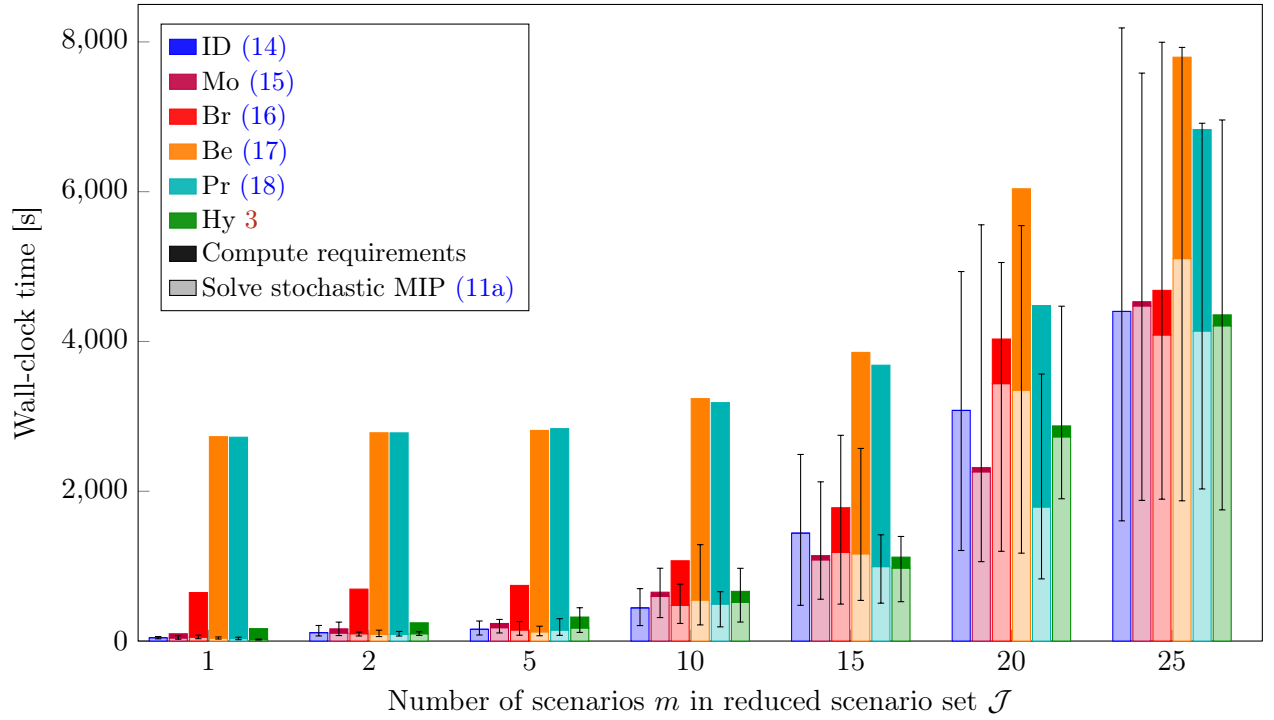

\section{Summary and conclusions}
\label{sec:conclusions}
In this paper, we revisit the classical scenario reduction theory for stochastic programming, which measures distribution similarity using the optimal mass transportation problem. For that, we analyze several optimization-problem-driven cost functions suggested in the literature and propose a new one. When combining the proposed cost function~\eqref{eq:c_Pr} with the Forward Selection Algorithm~\citep{Dupacova2003} to find reduced sets of scenarios, we state and prove a theorem that guarantees the first scenario drawn from the initial sample distribution is optimal with respect to the RAE. We further present a novel hybrid algorithm that uses the cost function of \citet{Morales2009} for a fast initial pruning of the scenario set, followed by a refined scenario selection based on the proposed cost function applied to the pre-selected subset.

To evaluate the performance of cost functions and the hybrid algorithm for scenario reduction, we use a classical two-stage stochastic MIP problem that has been extensively analyzed in the literature: the stochastic unit commitment problem. Based on a small 24-bus and a large 300-bus system, with 200 and 1000 scenarios in the initial distribution, respectively, we find that the proposed cost function outperforms the others in terms of solution quality. In particular, with only one scenario, it achieves relative approximation errors of \qty{11.3}{\percent} and \qty{2.4}{\percent} for the small and large cases, respectively. Furthermore, the proposed cost function more robustly selects good scenarios as the initial probability distribution varies, whereas the other cost functions exhibit large variations. Finding accurate scenario sets supported by only a few atoms is particularly crucial in stochastic MIPs, as solution times increase drastically with the number of scenarios.

The proposed cost function is, however, computationally expensive to evaluate, undermining the trade-off when comparing overall wall-clock times to the cost function of \citet{Morales2009}, which quickly finds decent reduced scenario sets.
This drawback is overcome by our proposed Hybrid Algorithm~\cref{alg:Hybrid}, which reduces wall-clock time by a factor of 18 and work (as measured by Gurobi) by a factor of 66 compared to using the proposed cost function, without compromising solution accuracy. The algorithm therefore provides an excellent method for scenario reduction for two-stage stochastic MIPs.

Future work should focus on extending the proposed cost function and the hybrid algorithm to multi-stage stochastic MIPs and investigating their performance for settings where the second stage is nonconvex, e.g., because it itself contains integer variables. For cases where the second stage is convex, one may analyze if and how duality information can be leveraged to further improve the cost functions. Finally, it should be analyzed how best to use information from the scenario reduction process to derive a lower bound for the stochastic MIP related to the full distribution, thereby guaranteeing the overall performance of the scenario reduction process.

\appendix
\section{Stochastic unit commitment problem}
\label{sec:app_SUC}

As an exemplary application of the general two-stage stochastic MIP Problem~\eqref{eq:GeneralTwoStageProblem}, we use the well-known two-stage SUC problem with uncertain renewable power production and follow the formulation presented by~\citet{Carrion2006} and \citet{Blanco2017}.
In this context, first-stage variables $x$ are the binary commitment decisions of dispatchable thermal units, which are subject to a set of technical constraints $\mathcal{X}$, e.g, minimum up- and downtimes, chosen to minimize a piece-wise linear cost function that takes, e.g., start-up and fixed costs into account. The second-stage problem minimizes the cost of operating the power system, ensuring technical feasibility, e.g., by restricting power production from generators or maximum line power flows, after observing the realization of uncertain renewable power production~$h^{\xi_i}$. 
In the following, we introduce the optimization problem and use lower-case letters for variables and upper-case letters for parameters, unless otherwise indicated.

Let $n, m \in \mathcal{N}$, $g \in \mathcal{G}$, $j \in \mathcal{J}$, and $t \in \mathcal{T} = \{1,2,...,T\}$ denote the set of nodes (buses), dispatchable generators, wind power generators, and time periods, respectively. Furthermore, let $\omega \in \Omega$ denote the set of wind power scenarios with associated power production~$\widetilde{W}_{j,t,\omega}$.

We use $T^{U^0}_g$ to denote the number of time steps that generator $g$ has already been operating in its state $U^0_g$ at $t = 1$ and $\mathcal{T}^{0}_g$ for the number of time steps that a generator $g$ has to remain in its initial state to account for minimum up and down times $\UT$ and $\DT$, respectively, as:
\begin{equation}
    T^{0}_g =
    \begin{cases}
        \min (T, \UT_g - T^{U^0}_g) & \text{if}~U^0_g = 1, \\
        \min (T, \DT_g - T^{U^0}_g) & \text{if}~U^0_g = 0.
    \end{cases}
\end{equation}

For example, if a generator $g'$ has been offline for 8 time steps prior to $t = 1$, such that $U^0_g = 0$ and $T^{U^0}_{g'} = 8$, and its minimum down time is $\DT_{g'} = 12$ then $T^{0}_g = 4$ and $\mathcal{T}^{0}_{g'} = \{1,2,3,4\}$.

The goal of the unit commitment problem is to optimally choose commitment decisions $u_{g,t}$ (on/off), start-ups $y_{g,t}$. and shut-downs $z_{g,t}$ for each dispatchable generator $g$ and time step $t$. Here, the objective function of the first-stage problem is given by:
\begin{equation}
    \sum_{t \in \mathcal{T}} \sum_{g \in \mathcal{G}} \left( C_g^{\mathrm{SU}} y_{g,t} + C_g^{\mathrm{fix}} u_{g,t} \right), \label{eq:FirstStage_Objective}
\end{equation}
where $C_g^{\mathrm{SU}}$ and $C_g^{\mathrm{fix}}$ refer to the start-up and fix costs, respectively.

The feasible region of the first-stage problem is given by:
\begin{subequations}\label{eq:FirstStage_feasible_region}
\begin{align}
& y_{g,t} - z_{g,t} = u_{g,t} - u_{g,t-1}, & & \forall g , \forall t \in \{2,...,T\}, \label{eq:FirstStage_state_change} \\
& y_{g,t} - z_{g,t} = u_{g,t} - U_g^0, & & \forall g, t = 1, \label{eq:FirstStage_state_change_init} \\
& y_{g,t} + z_{g,t} \leq 1, & & \forall g , \forall t \in \mathcal{T}, \label{eq:FirstStage_state_change_single} \\
& u_{g,t} = U_g^0 , & & \forall g , \forall t \in \mathcal{T}^0_g, \label{eq:FirstStage_state_initial}  \\
& \sum_{\tau = t - \mathrm{UT}_g + 1}^t y_{g,t} \leq u_{g,t}, & & \forall g, \forall t \in \mathcal{T} \setminus \mathcal{T}^0_g, \label{eq:FirstStage_min_up_time} \\
& \sum_{\tau = t - \mathrm{DT}_g + 1}^t z_{g,t} \leq 1 - u_{g,t}, & & \forall g, \forall t \in \mathcal{T} \setminus \mathcal{T}^0_g, \label{eq:FirstStage_min_down_time}
\end{align}
\end{subequations}
where $x = \{  (u_{g,t}, y_{g,t}, z_{g,t}), \forall g, \forall t\}$. Constraints~\eqref{eq:FirstStage_state_change}--\eqref{eq:FirstStage_state_change_init} relate start-ups and shut-downs to on/off states, Constraints~\eqref{eq:FirstStage_state_change_single} forbid simultaneous start-up and shut-down, Constraints~\eqref{eq:FirstStage_state_initial} enforce the initial state to an exogenously given one $U_g^0$, and Constraints~\eqref{eq:FirstStage_min_up_time}--\eqref{eq:FirstStage_min_down_time} define minimum up and down times, respectively.

The objective function for the second-stage problem of the SUC problem for each scenario $\omega \in \Omega$ is given by:
\begin{equation}\label{eq:SecondStage_objective}
   \sum_{t \in \mathcal{T}} \left( \sum_{g \in \mathcal{G}} C_g^{\mathrm{L}} p_{g,t,\omega} + \sum_{d \in \mathcal{D}} C^{\mathrm{Sh}} l^{\mathrm{Sh}}_{d,t,\omega} + \sum_{n \in \mathcal{N}} C_n^{\mathrm{Sl}} l^{\mathrm{Sl}}_{n,t,\omega} \right),
\end{equation}
where$p_{g,t,\omega}$ is the power production of dispatchable generator $g$ in time step $t$ and scenario $\omega$ at cost $C_g^{\mathrm{L}}$, $l^{\mathrm{Sh}}_{d,t,\omega}$ is the load shedding of demand $d$ at uniform cost $C^{\mathrm{Sh}}$, and $l^{\mathrm{Sl}}_{n,t,\omega}$ is a nodal excess slack demand that ensures feasibility for every first-stage unit commitment decision at cost $C_d^{\mathrm{Sl}}$, which we assume to be equal to the uniform load shedding cost $C^{\mathrm{Sh}}$.

The feasible region for every scenario $\omega \in \Omega$ is given by:
\begin{subequations}\label{eq:SecondStage_constraints}
\allowdisplaybreaks
\begin{align}
& P_g^{\min} u_{g,t} \leq p_{g,t,\omega} \leq P_g^{\max} u_{g,t}, & & \forall g, \forall t \in \mathcal{T}, \label{eq:SecondStage_prod_limit} \\
& -\RD_g \leq p_{g,t,\omega} - p_{g,t-1,\omega}, & & \forall g, \forall t \in \{2,...,T\}, \label{eq:SecondStage_ramp_limit_down} \\
& p_{g,t,\omega} - p_{g,t-1,\omega} \leq \RU_g, & & \forall g, \forall t \in \{2,...,T\}, \label{eq:SecondStage_ramp_limit_up} \\
& (P_g^{0} + \RD_g) u_{g,t} \leq p_{g,t,\omega}, & & \forall g, t = 1, \label{eq:SecondStage_ramp_limit_init_down} \\
& p_{g,t,\omega} \leq (P_g^{0} + \RU_g) u_{g,t}, & & \forall g, t = 1, \label{eq:SecondStage_ramp_limit_init_up} \\
& 0 \leq l^{\mathrm{Sh}}_{d,t,\omega} \leq L_{d,t}, & & \forall d, \forall t \in \mathcal{T}, \label{eq:SecondStage_shed_limit} \\
& 0 \leq w^{\mathrm{Sp}}_{j,t,\omega} \leq W_{j,t,\omega}, & & \forall j, \forall t \in \mathcal{T}, \label{eq:SecondStage_spill_limit} \\
& f_{n,m,t,\omega} = B_{n,m} (\theta_{n,t,\omega} - \theta_{m,t,\omega}), & & \forall n, \forall m \in \mathcal{N}_n, \forall t \in \mathcal{T}, \label{eq:SecondStage_flow_def} \\
& -\mathrm{F}_{n,m}^{\max} \leq f_{n,m,t,\omega} \leq \mathrm{F}_{n,m}^{\max}, & & \forall n, \forall m \in \mathcal{N}_n, \forall t \in \mathcal{T},\label{eq:SecondStage_line_limit} \\
& \sum_{g \in \mathcal{G}_n} p_{g,t,\omega}  - \sum_{m \in \mathcal{N}_n} f_{n,m,t,\omega} - & &  \nonumber \\
& l^{\mathrm{Sl}}_{n,t,\omega} + \sum_{j \in \mathcal{J}_n} (\widetilde{W}_{j,t,\omega} - w^{\mathrm{Sp}}_{j,t,\omega}) - & & \nonumber \\
& \sum_{d \in \mathcal{D}_n} (L_{d,t} - l^{\mathrm{Sh}}_{d,t,\omega})  = 0, & & \forall n, \forall t \in \mathcal{T}, \label{eq:SecondStage_nodal_bal}
\end{align}
\end{subequations}
where Constraints~\eqref{eq:SecondStage_prod_limit} limit power production of dispatchable generator $g$ in time step $t$ and scenario $\omega$ to be between the minimum and maximum technical limits $P_g^{\min}$ and $P_g^{\max}$, respectively. Constraints~\eqref{eq:SecondStage_ramp_limit_down}--\eqref{eq:SecondStage_ramp_limit_init_up} restrict their ramping capabilities to between a maximum ramp-down and ramp-up limit $\RD_g$ and $\RU_g$, respectively. Load shedding $l^{\mathrm{Sh}}_{d,t,\omega}$ of demand $d$ is restricted to be below the exogenously given inflexible load $L_{d,t}$ by Constraints~\eqref{eq:SecondStage_shed_limit}. Constraints~\eqref{eq:SecondStage_spill_limit} limit the amount of wind power spillage $w^{\mathrm{Sp}}_{j,t,\omega}$ to be below the respective forecast in scenario $\omega$, $\widetilde{W}_{j,t,\omega}$. We use a loss-less DC power flow approximation to define line flows $f_{n,m,t,\omega}$ based on line susceptance $B_{n,m}$ and nodal voltage angles $\theta_{n,t,\omega}$ through Constraints~\eqref{eq:SecondStage_flow_def}. Symmetric line flow limits are enforced by Constraints~\eqref{eq:SecondStage_line_limit} and nodal power balance is ensured by Constraints~\eqref{eq:SecondStage_nodal_bal}.

The second-stage operational problem of the SUC for every scenario $\omega \in \Omega$ is then given by:
\begin{equation}\label{eq:SecondStage}
    \min_{} \eqref{eq:SecondStage_objective}, \mathrm{s.t.} \eqref{eq:SecondStage_constraints}.
\end{equation}
Note that, depending on whether the second-stage problem is solved without the first-stage problem, $u_{g,t}$ will be a parameter, not a variable.

The full two-stage SUC problem is then given by:
\begin{equation}\label{eq:TwoStageUnitCommitment}
\min_{x \in \eqref{eq:FirstStage_feasible_region}} \eqref{eq:FirstStage_Objective} + \mathbb{E}_{\omega}[\eqref{eq:SecondStage}] = \min_{x \in \eqref{eq:FirstStage_feasible_region}} \eqref{eq:FirstStage_Objective} + \sum_{\omega \in \Omega} \pi_\omega \cdot \eqref{eq:SecondStage}, 
\end{equation}
where the equality follows from the assumption that $\omega$ follows a finite discrete distribution where $\pi_\omega$ denotes the probability of scenario $\omega$.

\section{Proof of Theorem 1}
\label{sec:app_proof}

We first note that the term $z^*(\mathbb{P}_n)$ in the definition of the $\RAE$ in~\eqref{eq:RAE} only depends on the distribution $\mathbb{P}_n$ and not on any set of scenarios $\mathcal{J}$ drawn from it. Consequently, the solution $x^*(\mathbb{Q}_m)$, defined by Equation~\eqref{eq:TwoStageOptSolQ}, which minimizes the first term in the numerator, also minimizes the $\RAE$. Secondly, there exists only one unique probability distribution carried on a single atom: the Dirac measure that places unit probability mass on that same atom, i.e.,  $\mathbb{Q}_1^j = \delta_{\xi_j}, \forall j \in \mathcal{I}$. Now observe that in this case whenever the distribution of $\xi$ is supported solely on a single scenario $\xi_j$, such that $\xi \sim \mathbb{Q}_1^j = \delta_{\xi_j}$ and $p_j = 1$, the general two-stage stochastic problem~\eqref{eq:GeneralTwoStageProblem} and the single-scenario problem~\eqref{eq:GeneralSingleScenarioDP} are identical. Therefore, their optimal solutions $x^*(\mathbb{Q}_1^j)$ and $x^*(\xi_j)$, defined by Equations~\eqref{eq:TwoStageOptSol} and \eqref{eq:SingleScenarioOptSol}, respectively, must also be identical, i.e., $x^*(\mathbb{Q}_1^j) = x^*(\xi_j)$. Based on those points and assuming\footnote{While we find this assumption to be naturally fulfilled in most problems, it can always be ensured by adding a large enough positive constant to the objective function.} that $z^*(\mathbb{P}_n) > 0$, we derive:
\begin{align}
    \mathbb{Q}_1^{j^*} &= \arg\min_{\mathbb{Q}_1^j \in \hat{\mathcal{Q}}_1} \RAE(\mathbb{P}_n, \mathbb{Q}_1^j) \nonumber \\
    &= \arg\min_{\mathbb{Q}_1^j \in \hat{\mathcal{Q}}_1} \frac{z(x^*(\mathbb{Q}_1^j),\mathbb{P}_n) - z^*(\mathbb{P}_n)}{z^*(\mathbb{P}_n)} \nonumber \\
    &= \arg\min_{\mathbb{Q}_1^j \in \hat{\mathcal{Q}}_1} z(x^*(\mathbb{Q}_1^j),\mathbb{P}_n) \nonumber\\
    &= \arg\min_{j \in \mathcal{I}} z(x^*(\xi_j),\mathbb{P}_n) \nonumber \\
    &= \arg\min_{j \in \mathcal{I}} \sum_{i \in \mathcal{I}} p_i z(x^*(\xi_j),\xi_i), \label{eq:RAEReformulationZeta}
\end{align}
where the last equivalence follows from Equation~\eqref{eq:SecondStageEvaluationDistribution}.

Now taking a look at the first scenario drawn from $\mathcal{I}$ using the Forward Selection Algorithm~\cref{alg:ForwardSelection}, such that $\mathcal{J} = \{\}$ and $\mathcal{R} = \mathcal{I}$, Line~4 can be rewritten as:
\begin{align}
    j &= \arg\min_{j' \in \mathcal{R}} \sum_{i \in \mathcal{R} \setminus \{ j' \}} p_i \min_{j'' \in \mathcal{J} \cup \{j'\} } c(\xi_i,\zeta_{j''}) \nonumber\\
    &= \arg\min_{j' \in \mathcal{I}} \sum_{i \in \mathcal{I} \setminus \{ j' \}} p_i \min_{j'' \in \{\} \cup \{j'\} } c(\xi_i,\zeta_{j''}) \nonumber \\
    &= \arg\min_{j' \in \mathcal{I}} \sum_{i \in \mathcal{I} \setminus \{ j' \}} p_i c(\xi_i,\zeta_{j'}) \nonumber \\
    &= \arg\min_{j' \in \mathcal{I}} \sum_{i \in \mathcal{I}} p_i c(\xi_i,\zeta_{j'}), \label{eq:FSReformulation}
\end{align}
where the last equivalence follows from $c(\xi_{j'},\zeta_{j'}) = 0$.

Finally, we observe that for the measure $c^{\mathrm{Pr}}$, the second term, $z(x(\xi_i),\xi_i))$, is independent of scenario $j$, and, similar to the $\RAE$, is therefore identical for every scenario $j$. Now, replacing the alias $j'$ with $j$ and setting $c(\xi_i,\zeta_{j}) =  c^{\mathrm{Pr}}(\xi_i,\zeta_{j})$ in Equation~\eqref{eq:FSReformulation} we obtain:
\begin{align}
    j^* &= \arg\min_{j \in \mathcal{I}} \sum_{i \in \mathcal{I}} p_i c^{\mathrm{Pr}}(\xi_i,\zeta_{j}) \nonumber \\
    &= \arg\min_{j \in \mathcal{I}} \sum_{i \in \mathcal{I}} p_i (z(x^*(\zeta_{j}), \xi_i) - z(x^*(\xi_i),\xi_i)) \nonumber \\
    &= \arg\min_{j \in \mathcal{I}} \sum_{i \in \mathcal{I}} p_i z(x^*(\zeta_{j}), \xi_i), \label{eq:FSOptimalityFirstScenario}
\end{align}
which is equivalent to the minimization of the $\RAE$ in Equation~\eqref{eq:RAEReformulationZeta} and completes the proof. $\qedsymbol$
\section{Computational configurations}
\label{sec:app_configs}

For the cost function $c^{\mathrm{Mo}}$, we solve the EVP (MIP) on four cores, while for $c^{\mathrm{Br}}$, $c^{\mathrm{Be}}$, and $c^{\mathrm{Pr}}$, we solve all single-scenario MIPs in parallel on one core each. To facilitate the solution of many similar MILP problems, we first reduce the set of scenarios $\mathcal{I}$ to 192 scenarios (one per core) using the cost function $c^{\mathrm{ID}}$, and then solve the corresponding MIPs. We then use the solutions to warm-start the MIPs for the remaining scenarios using their closest center.
For the Hybrid Algorithm~\cref{alg:Hybrid} with $r=50$, we solve each MILP on three to four cores based on the availability of the 192 maximum available cores. The second-stage DC OPF problems are solved in parallel on one core each using dual simplex, with in-place modification of the right-hand-side scenario realizations and reoptimization.

Additionally, we warm-start the reduced two-stage stochastic MIP problems with the best integer solution obtained by solving the single-scenario problems required to evaluate each cost function shown in Table~\ref{tab:OptProbRequirments} and for the Hybrid Algorithm~\cref{alg:Hybrid}. In particular, we use the integer solution of the EVP for $c^{\mathrm{Mo}}$, and the one of the scenario drawn first by the Forward Selection Algorithm~\ref{alg:ForwardSelection} for each of the other metrics (the most central scenario with respect to each cost function) and of Phase 2 of the Hybrid Algorithm~\ref{alg:Hybrid}.

\section*{AI Usage Disclosure}
Anthropic Claude (Opus 4.5 -- 4.8), accessed through both the desktop application and the Claude Code terminal interface, was used for code-related activities, including software design brainstorming, implementation support, and figure creation. All AI-generated output was carefully reviewed, tested, and frequently modified before inclusion. The authors take full responsibility for the code.
\section*{Acknowledgments}
The computational results have been achieved using the Austrian Scientific Computing (ASC) infrastructure. \\
We thank Bhumi Kumar, Florian Schimek, and Ruben van Beesten for valuable discussions that helped us to improve our work and this manuscript. \\
The work of J. M. Morales and S. Pineda is supported by the Spanish Ministry of Science, Innovation and Universities through project PID2023-148291NB-I00 and by the Department of Universities, Research and Innovation of the Regional Government of Andalusia through project DGP-PIDI-2024-00851.

\bibliographystyle{elsarticle-harv} 
\bibliography{SR_for_UC}

\end{document}